\documentclass[a4paper]{amsart}

\usepackage{mathtools, amsthm, amssymb, hyperref, bbm}
\usepackage[capitalise, nameinlink, noabbrev, nosort]{cleveref} 
\usepackage{xy} \xyoption{all}

\theoremstyle{plain}
\newtheorem{lma}{Lemma}[section]
\crefname{lma}{Lemma}{Lemmata}
\newtheorem{thm}[lma]{Theorem}
\crefname{thm}{Theorem}{Theorems}
\newtheorem{cor}[lma]{Corollary}
\crefname{cor}{Corollary}{Corollaries}
\newtheorem{prp}[lma]{Proposition}
\crefname{prp}{Proposition}{Propositions}

\theoremstyle{definition}
\newtheorem{pgr}[lma]{}
\crefname{pgr}{Paragraph}{Paragraphs}
\newtheorem{dfn}[lma]{Definition}
\crefname{dfn}{Definition}{Definitions}

\newcounter{theoremintro}
\newtheorem{thmintro}[theoremintro]{Theorem}

\theoremstyle{remark}
\newtheorem{rmk}[lma]{Remark}
\crefname{rmk}{Remark}{Remarks}

\crefname{rmks}{Remarks}{Remarks}

\crefname{exa}{Example}{Examples}

\crefname{qst}{Question}{Questions}
\newtheorem{con}[lma]{Conjecture}
\crefname{con}{Conjecture}{Conjectures}

\newcommand{\calG}{\mathcal{G}}
\newcommand{\calH}{\mathcal{H}}

\newcommand{\Bo}{\mathcal{B}^{\mathrm{o}}}
\newcommand{\Bco}{\mathcal{B}_{\mathrm{c}}^{\mathrm{o}}}
\newcommand{\etale}{\'etale}
\newcommand{\SInv}{\mathfrak{S}_{\mathrm{pi}}}
\newcommand{\ran}{\mathsf{r}}
\newcommand{\sou}{\mathsf{s}}

\newcommand{\Fp}{F_{\lambda}^{p}}
\newcommand{\Fq}{F_{\lambda}^{q}}
\newcommand{\Fpast}{F_{\lambda}^{p, \ast}}
\newcommand{\LIG}{L^{I}}

\newcommand{\PI}{\mathrm{PI}}
\newcommand{\PIMP}{\mathrm{PI}_{\mathrm{MP}}}
\newcommand{\subMP}{\mathrm{MP}}
\newcommand{\subHerm}{\mathrm{h}}
\newcommand{\ca}{$\rm C^*$-algebra}
\newcommand{\ba}{Banach algebra}
\newcommand{\lpoa}{$L^p$-operator algebra}
\newcommand{\andSep}{\,\,\,\text{ and }\,\,\,}
\newcommand{\TT}{\mathbb{T}}
\newcommand{\CC}{\mathbb{C}}
\newcommand{\RR}{\mathbb{R}}

\newcommand{\NN}{\mathbb{N}}
\newcommand{\Bdd}{\mathcal{B}}

\DeclareMathOperator{\supp}{supp}
\DeclareMathOperator{\core}{core}
\DeclareMathOperator{\linSpan}{span}

\numberwithin{equation}{section} 

\title{Rigidity of pseudofunction algebras of ample groupoids} 
\date{\today}

\author{Eusebio Gardella}
\author{Mathias Palmstr{\o}m}
\author{Hannes Thiel}

\address{Eusebio Gardella,
	Department of Mathematical Sciences, Chalmers University of
	Technology and University of Gothenburg, Gothenburg SE-412 96, Sweden.}
\email{gardella@chalmers.se}
\urladdr{www.math.chalmers.se/~gardella}

\address{Department of Mathematical Sciences, Faculty of Information Technology and Electrical Engineering, NTNU -- Norwegian University of Science and Technology, Trondheim, Norway}
\email{mathias.palmstrom@ntnu.no}
\urladdr{www.ntnu.edu/employees/mathias.palmstrom}

\address{Hannes~Thiel, 
	Department of Mathematical Sciences, Chalmers University of Technology and the University of
	Gothenburg, SE-412 96 Gothenburg, Sweden}
\email{hannes.thiel@chalmers.se}
\urladdr{www.hannesthiel.org}

\thanks{
	EG was partially supported by the Swedish Research Council Grant 2021-04561.
	HT was partially supported by the Knut and Alice Wallenberg Foundation (KAW 2021.0140).
}

\subjclass[2020]%
{Primary
	47L10; 
	Secondary
	46L55, 
	20M18. 
}

\begin{document}

\begin{abstract}
We show that a Hausdorff, ample groupoid $\calG$ can be completely recovered from the $I$-norm completion of $C_c(\calG)$.
More generally, we show that this is also the case for the algebra of symmetrized $p$-pseudofunctions, as well as for the reduced groupoid $L^p$-operator algebra, for $p\neq 2$.

Our proofs are based on a new construction of an inverse semigroup built from Moore-Penrose invertible partial isometries in an $L^p$-operator algebra.
Along the way, we verify a conjecture of Rako\v{c}evi{\'c} concerning the continuity of the Moore-Penrose inverse for $L^p$-operator algebras.
\end{abstract}	

\maketitle

\renewcommand*{\thetheoremintro}{\Alph{theoremintro}}

\section{Introduction} 
\label{sec: introduction}

Wendel's celebrated theorem (\cite{Wen51IsometrIsoGpAlgs}) is a foundational result in 
abstract harmonic analysis that highlights a deep connection between the algebraic structure of a group and the functional-analytic properties of its associated $L^1$-group algebra. 
The theorem is both powerful and easy to state: if $G$ and $H$ are locally compact groups, then there is an isometric isomorphism\footnote{Although Wendel's original theorem assumes that a $\ast$-isomorphism exists, it was shown recently in 
\cite{GarThi22IsoConv} that this assumption is not necessary.} of Banach algebras
$L^1(G)\cong L^1(H)$ if and only if there is an isomorphism of topological 
groups $G\cong H$. 
This result reveals a remarkable rigidity: the Banach-algebraic structure of $L^1(G)$ completely encodes the group-theoretic data of $G$. 

Wendel's result laid the groundwork for a broader attempt to understand how certain properties of a group (such as amenability) are reflected in its $L^1$-algebra. 
It also motivated the study of similar rigidity questions for other Banach algebras associated to groups. 
For example, Johnson showed that the measure algebra $M(G)$ of a group $G$ completely remembers~$G$. 
For group \ca{s}, the situation is more complicated, and it is known that there may be an isomorphism $C^*_\lambda(G)\cong C^*_\lambda(H)$ for non-isomorphic groups $G$ and $H$; this 
is already the case for $G=\mathbb{Z}_4$ and $H=\mathbb{Z}_2\oplus \mathbb{Z}_2$.
For the algebras of $p$-pseudofunctions introduced by Herz \cite{Her71pSpApplConv}
and recently studied extensively (see, for example, \cite{HejPoy15SimpleLp, GarThi15GpAlgLp, GarThi19ReprConvLq, Phi19arX:SimpleRedGpBAlgs}), an analog of Wendel's theorem was obtained in \cite{GarThi22IsoConv}: for $p\neq 2$, there exists an isometric isomorphism $F^p_\lambda(G)\cong F^p_\lambda(H)$ if and only if $G$ and $H$ are isomorphic. 
In this work, we are interested in pursuing generalizations of Wendel's theorem where groups are replaced by more general algebraic structures.

Groupoids arise naturally as unifying objects for capturing symmetries in mathematical structures that are too intricate to be described by groups alone, and generalize both groups and equivalence relations. 
While groups encode global symmetries, groupoids are well-suited to situations where transformations are only local or partially defined, making them useful tools in many areas of mathematics, such as dynamical systems, foliation theory, and noncommutative geometry. 
The convolution algebra $L^I(\calG)$ associated to the $I$-norm of a groupoid $\calG$ generalizes the group algebra $L^1(G)$, and this construction serves as a bridge for applying Banach-algebraic methods to the study of geometric and dynamical data.
In light of the deep impact of Wendel's theorem, it is a very natural problem to determine whether $L^I(\calG)$ completely remembers $\calG$. 

In this work, we tackle the isomorphism problem for the convolution algebras associated to Hausdorff groupoids. 
While most of our results apply in the \'etale case, our strongest result is for ample groupoids (see \cref{pgr:groupoids}):

\begin{thmintro} [See \cref{prp:LIRigidity}]
Let $\calG$ and $\calH$ be Hausdorff, ample groupoids. 
Then there is an isometric isomorphism $\LIG (\calG) \cong \LIG (\calH)$ if and only if $\calG \cong \calH$.
\end{thmintro}

We emphasize that we do not need the isomorphism $\LIG (\calG) \cong \LIG (\calH)$ to preserve the involution. 

Our strategy to prove the above theorem is to show that the inverse semigroup~$\Bo_c(\calG)$ of compact, open bisections of a Hausdorff, ample groupoid $\calG$ is entirely encoded in $\LIG (\calG)$. 
Once this is accomplished, the desired rigidity follows from a reconstruction result of Exel; see \cite[Theorem 4.8]{Exe10ReconstrTotDiscGpds}. 

Despite $L^I(\calG)$ not being itself representable on an $L^p$-space, the recent advances in the study of $L^p$-operator algebras give a hint of how the unit space $\calG^{(0)}$ may be encoded in $L^I(\calG)$, and we will show that $C_0(\calG^{(0)})$ is isomorphic to the C*-core of $L^I(\calG)$, namely the complex vector space generated by the hermitian elements in $L^I(\calG)$; see \cref{pgr:core}.
The way we will identify the compact open bisections of $\calG$ within $L^I(\calG)$ is using (homotopy classes of) what we call \emph{Moore-Penrose invertible partial isometries}, or MP-partial isometries for short; see \cref{pgr:PI}. 
These are contractive elements $a\in L^I(\calG)$ such that there exists a (necessarily unique) contraction $b\in L^I(\calG)$ (called the Moore-Penrose inverse) with 
\[
a=a b a \ \ \mbox{ and } \ \ b = bab
\]
such that $ab$ and $ba$ are hermitian.

It is not too difficult to show that $C_0(\calG^{(0)})$ is contained in the C*-core of $L^I(\calG)$, and that compact, open bisections in $\calG$ induce MP-partial isometries in $L^I(\calG)$. 
On the other hand, both converse directions require a fair amount of work.
For once, it is not a priori clear that the C*-core of $L^I(\calG)$ is closed under multiplication (and hence indeed a C*-algebra), since it is not obvious why hermitian elements in $L^I(\calG)$ should commute. 
It is even more challenging to show that the set $\PIMP(L^I(\calG))$ of MP-partial isometries in $L^I(\calG)$ is an inverse semigroup under multiplication. 

In order to circumvent this, we will regard $L^I(\calG)$ as the case $p=1$ of the more general construction of the \emph{symmetrized $p$-pseudofunction algebra}, $F^{p,\ast}_\lambda(\calG)$. 
For $p\in [1,\infty)$, there is a canonical representation $\varphi_p\colon F^{p,\ast}_\lambda(\calG)\to \mathcal{B}(\ell^p(\calG))$ which extends left convolution on $C_c(\calG)$.
Although this representation is not isometric, it is contractive and injective (\cref{prp:FpStar-Fp-Injective}), and the closure of the image can be identified with the reduced groupoid $L^p$-operator algebra $F^p_\lambda(\calG)$, which has been introduced and studied in \cite{GarLup17ReprGrpdLp, ChoGarThi24LpRigidity}. 
As it turns out, $\varphi_p$ induces an identification of the MP-partial isometries of $F^{p,\ast}_\lambda(\calG)$ and of $F^{p}_\lambda(\calG)$. 
The crucial advantage in this approach is that both the hermitian elements and the MP-partial isometries of Banach algebras that act on an $L^p$-space, for $p\neq 2$, can be described by exploiting Lamperti's description of the invertible isometries between $L^p$-spaces. 

The ultimate outcome is that $C_0(\calG^{(0)}$ is isometrically isomorphic to the C*-core of $F^{p,\ast}_\lambda(\calG)$, and that the inverse semigroup $\Bo_c(\calG)$ of compact, open bisections in~$\calG$ is naturally isomorphic to the inverse semigroup $\PIMP(F^{p,\ast}_\lambda(\calG))/\!\simeq$ of homotopy classes of MP-partial isometries in $F^{p,\ast}_\lambda(\calG)$, for any $p\in [1,\infty)\setminus\{2\}$. 

Our methods also allow us to prove a rigidity result for both $F^{p}_\lambda(\calG)$ and its symmetrized version $F^{p,\ast}_\lambda(\calG)$: 

\begin{thmintro}[See \cref{prp:LpRigidity} and \cref{prp:SymLpRigidity}]
Let $\calG$ and $\calH$ be Hausdorff, ample groupoids, and let $p\in [1,\infty)\setminus\{2\}$. Then the following are equivalent:
\begin{enumerate}
 \item There is an isometric isomorphism $F^p_\lambda (\calG) \cong F_{\lambda}^{p} (\calH)$.
  \item There is an isometric isomorphism $F^{p,\ast}_\lambda (\calG) \cong F_{\lambda}^{p,\ast} (\calH)$.
  \item There is a topological groupoid isomorphism $\calG\cong \calH$.
\end{enumerate}
\end{thmintro}

For \emph{topologically principal} groupoids, the theorem above recovers \cite[Corollary~5.6]{ChoGarThi24LpRigidity}, although the proofs are dramatically different since we do not use any kind of normalizers in this work, and the C*-core only plays a minor role in our methods. 
On the other hand, for groupoids
with trivial unit space, the theorem above recovers the main result of \cite{GarThi22IsoConv} for discrete groups. 

It should be mentioned that there is a vast amount of literature around the problem of reconstructing a groupoid $\calG$ from the inclusion $C_0(\calG^{(0)})\hookrightarrow C^*_\lambda(\calG)$. 
This line of research was initiated with Renault's celebrated work \cite{Ren08Cartan} on reconstruction of (topologically principal) groupoids from a Cartan subalgebra, which was recently strengthened in \cite{CarRuiSimTom21ReconGpdsCstarRigidity}. 
Algebraic analogs (for example, where $C^*_\lambda(\calG)$ is replaced by the Steinberg algebra of $\calG$) have also been explored by a number of authors; 
see \cite{Ste19DiagPresIsoGpdAlg, BicSta19GenNCStoneDuality, BicCla21ReconstrEtaleGpd}. 
We emphasize that our rigidity results do not assume that we know beforehand the position of $C_0(\calG^{(0)})$ in our Banach algebras. 
Even if this is ignored, our results apply to a class of groupoids which is considerably larger than those considered in previous works.

There is an additional upshot of our methods:
the tools and machinery developed to prove the above rigidity results
allows us to establish, for a large class of Banach algebras, Rako\v{c}evi{\'c}'s conjecture (\cite{Rak88MPInvBAlg}; see \cref{cnj}) on the continuity of Moore-Penrose inversion for partial isometries:

\begin{thmintro}[See \cref{prp:ContinuityMP}]
Let $A$ be a unital \ba{} whose hermitian idempotents are ultrahermitian and commute.
Then $A$ satisfies Rako\v{c}evi{\'c}'s conjecture. 
This applies in particular to $L^p$-operator algebras, for $p\in [1,\infty)$.
\end{thmintro}

\subsection*{Acknowledgements}

The authors thank Tristan Bice and Benjamin Steinberg for valuable comments on reconstruction results for groupoids.

\section{Inverse semigroups of partial isometries} 
\label{sec:PI}

In this section, we first extend the theory of $\rm C^*$-cores (\cref{pgr:core}), Moore-Penrose invertibility (\cref{pgr:MP}), and MP-partial isometries (\cref{pgr:PI}) from the usual setting of unital \ba{s} to approximately unital \ba{s}.
This will allow us to deal with algebras of groupoids whose unit spaces are not compact. 

We characterize when the MP-partial isometries in a Banach algebra 
form an inverse semigroup under multiplication (\cref{prp:CharPIMP-InvSgp}), and we show that a sufficient condition is that hermitian idempotents are ultrahermitian and commute (\cref{prp:PIMP-InvSgp}).
Under this assumption, we verify Rako\v{c}evi{\'c}'s conjecture (\cref{con:RakocvevicsConjecture}) on the continuity of Moore-Penrose inversion (\cref{prp:ContinuityMP}), and we show that the set $\SInv(A)$ of homotopy classes of MP-partial isometries (\cref{dfn:SInv}) inherits the structure of an inverse semigroup.
All of this applies in particular to \lpoa{s} (\cref{prp:InvSgpHtpyPIMP}) for $p \neq 2$.

\begin{pgr}[Approximately unital Banach algebras]
Let $A$ be a Banach algebra. 
The collection of all contractive elements in $A$ will be denoted by $A_1$. 
Following \cite{BleOza15RealPosApproxIdBAlgs}, 
we say that~$A$ is \emph{approximately unital} if it contains a contractive approximate identity, that is, a net $(e_j)_{j\in J}$ in $A_1$ such that
\[
\lim_j \| a - ae_j \| = 0, \andSep
\lim_j \| a - e_j a \| = 0
\]
for every $a \in A$.
Every unital Banach algebra (always assumed with a contractive unit element) is approximately unital, as is every \ca.
We say that a homomorphism $\varphi \colon A \to B$ of approximately unital Banach algebras is \emph{approximately unital} if there exists some contractive approximate identity for $A$ whose image under $\varphi$ is a contractive approximate identity for $B$. Note that a homomorphism between unital Banach algebras is approximately unital if and only if it is unital.
\end{pgr}

\begin{pgr}[Minimal multiplier unitization]
\label{pgr:unitization}
Let $A$ be an approximately unital \ba. If $A$ is unital, we set $\widetilde{A}=A$.
If $A$ is not unital,
following \cite[Definition~1.8]{BlePhi19LpOpAlgApproxId1} we let $\widetilde{A}$ be the \emph{minimal multiplier unitization} of $A$, which is defined as the Banach algebra $\widetilde{A} := A + \CC 1$ equipped with the norm 
\[
\| a + \mu 1 \|_{\widetilde{A}} := \sup_{b \in A_1} \| ab + \mu b \|_{A},
\]
for $a \in A$ and $\mu \in \CC$. 
It is easy to see that if $(e_j)_{j\in J}$ is a contractive approximate identity in $A$, then 
\[
\| a + \mu 1 \|_{\widetilde{A}} = \lim_{j\in J} \| a e_j + \mu e_j \|_{A} = \sup_{j\in J} \| a e_j + \mu e_j \|_{A} ,
\]
for $a \in A$ and $\mu \in \CC$.

One can view $\widetilde{A}$ as a unital, closed subalgebra of the algebra $\Bdd(A)$ of bounded linear operators $A \to A$, with $a \in A$ acting by left multiplication.
The canonical map $A \to \widetilde{A}$ is isometric, and identifies $A$ with a closed, two-sided ideal in $\widetilde{A}$ such that $\widetilde{A}/A \cong \CC$.

Every contractive, approximately unital homomorphism $A \to B$ between approximately unital \ba{s} extends naturally to a contractive, unital homomorphism $\widetilde{A} \to \widetilde{B}$.
\end{pgr}

\begin{pgr}[Hermitian elements]\label{pgr:hermitian}
Let $A$ be a Banach algebra.
If $A$ is unital, then $a \in A$ is \emph{hermitian} if $\| \exp(ita) \| = 1$ for all $t \in \RR$.
We use $A_\subHerm$ to denote the set of hermitian elements in~$A$.
If $A$ is approximately unital, then we follow \cite[Definition~2.8]{BlePhi19LpOpAlgApproxId1} and say that an element in~$A$ is \emph{hermitian} if it belongs to $A_{\rm h} := A \cap \widetilde{A}_{\rm h}$.

If $A$ is unital, then $A_\subHerm$ is a closed, real subspace, and $A_\subHerm + iA_\subHerm$ is a closed (complex) subspace such that every element has a unique representation as $a+ib$ with $a,b \in A_\subHerm$, which allows one to define an involution by $(a+ib)^* := a-ib$ for $a,b \in A_\subHerm$;
see \cite[Section~5]{BonDun71NumRanges}.
These statements generalize easily to the case that $A$ is approximately unital.

An element in a \ca{} is hermitian if and only if it is self-adjoint.
\end{pgr}

While $A_\subHerm + iA_\subHerm$ is always a closed subspace, it is in general not a subalgebra unless hermitian elements in $A$ are closed under multiplication. 

\begin{dfn}
\label{pgr:core}
Let $A$ be an approximately unital \ba, and assume that $A_\subHerm$ is closed under multiplication. We call the closed subalgebra $A_\subHerm + iA_\subHerm$ the \emph{C*-core} of $A$, and denote it by $\core(A)$.
\end{dfn}

Next, we show that $\core(A)$ is a commutative \ca.

\begin{prp}
\label{prop:CoreAbelian}
Let $A$ be an approximately unital \ba, and assume that $A_\subHerm$ is closed under multiplication. Then $\core(A)$ is a commutative \ca.
\end{prp}
\begin{proof}
First, we note that hermitian elements in $\widetilde{A}$ are closed under multiplication, and therefore $\core(\widetilde{A})$ is a unital, closed subalgebra of $\widetilde{A}$.
By the Vidav-Palmer theorem (\cite[Theorem~6.9]{BonDun71NumRanges}), it follows that $\core(\widetilde{A})$ is a \ca{}, that is, $\|x^*x\|=\|x\|^2$ for all $x \in \core(\widetilde{A})$.
It follows that elements in $\core(A)$ satisfy the $\rm C^*$-identity as well.
Since a \ca{} is commutative if (and only if) its self-adjoint elements are closed under multiplication, it follows that $\core(A)$ is commutative.
\end{proof}

\begin{rmk}
\label{rmk:ContrHomCore}
If $\varphi \colon A \to B$ is an approximately unital, contractive homomorphism between approximately unital \ba{s}, then $\varphi(A_\subHerm) \subseteq B_\subHerm$.
If hermitian elements in $A$ and $B$ are closed under multiplication, then $\varphi$ induces a $*$-homomorphism $\core(A) \to \core(B)$.
\end{rmk}

\begin{pgr}[Moore-Penrose invertible elements]
\label{pgr:MP}
Following Rako\v{c}evi{\'c} \cite{Rak88MPInvBAlg}, we say that an element~$a$ in a unital Banach algebra $A$ is \emph{Moore-Penrose invertible} if there exists $b \in A$ such that
\[
a = aba, \quad
b = bab, \andSep
ab,ba \in A_\subHerm.
\]
By \cite[Lemma~2.1]{Rak88MPInvBAlg}, an element $b$ with these properties is unique (if it exists), and it is called the \emph{Moore-Penrose inverse} of $a$, denoted by $a^\dag$.
We use $A_\subMP$ to denote the collection of all Moore-Penrose invertible elements in $A$.
	
Now let $A$ be an approximately unital \ba.
If an element $a \in A$ is Moore-Penrose invertible in $\widetilde{A}$, then $a^\dag$ automatically belongs to $A$ since $a^\dag = a^\dag a a^\dag$, and $A$ is an ideal in $\widetilde{A}$.
Therefore, an element $a \in A$ is Moore-Penrose invertible in $\widetilde{A}$ if and only if there exists $b \in A$ such that $a = aba$, $b = bab$, and $ab,ba \in A_\subHerm$ ($=A \cap \widetilde{A}_\subHerm$).
We will thus call an element in an approximately unital \ba{}~$A$ \emph{Moore-Penrose invertible} if it belongs to $A_\subMP := A \cap \widetilde{A}_\subMP$.
\end{pgr}

\begin{pgr}[Partial isometries]
\label{pgr:PI}
Following Mbekhta \cite[Definitions~4.1, 4.3]{Mbe04PartIsom}, a contractive operator~$a$ on a Banach space $E$ is called a \emph{partial isometry} if there exists a contractive operator $b$ on $E$ such that $a=aba$ and $b=bab$.
We note that the element~$b$ may not be uniquely determined. 
Further, $a \in \Bdd(E)$ is called a \emph{MP-partial isometry} if it is contractive and admits a contractive Moore-Penrose inverse in $\Bdd(E)$.
Generalizing this terminology, we introduce the following:\end{pgr}

\begin{dfn}
Let $A$ be an approximately unital \ba. We define the set of partial isometries (MP-partial isometries) in $A$ as
\begin{align*}
\PI(A) &:= \big\{ a \in A_1 : \text{there exists $b \in A_1$ such that $a=aba$ and $b=bab$} \big\}, \\
\PIMP(A) &:= \big\{ a \in A_\subMP : \|a\|, \|a^\dag\| \leq 1 \big\}.
\end{align*}
\end{dfn}

For \ca{s}, the above definition of a partial isometry agrees with the usual definition ($a=aa^*a$), and also with the notion of MP-partial isometry; see \cite[Theorem~3.1]{Mbe04PartIsom}.

The following is straightforward to check.

\begin{rmk}
\label{rem:ContrHomPIMP}
Let $\varphi\colon A\to B$ be an approximately unital, contractive homomorphism between approximately unital Banach algebras~$A$ and~$B$.
Then
\[
\varphi(\PIMP(A))\subseteq \PIMP(B).
\]
\end{rmk}

\begin{pgr}[Inverse semigroups]
Let $S$ be a semigroup.
An element $t \in S$ is called a \emph{generalized inverse} of an element $s \in S$ if $s=sts$ and $t=tst$.
One says that $S$ is an \emph{inverse semigroup} if every $s \in S$ admits a unique generalized inverse, denoted~$s^\dag$.
Then  $(st)^\dag = t^\dag s^\dag$ for all $s,t\in S$, and the map $s \mapsto s^\dag$ is an involution on $S$.
Further, given $s \in S$, the elements $ss^\dag$ and $s^\dag s$ are idempotent.
\end{pgr}

\begin{rmk}\label{rmk:CharInvSmgp}
It is known that a semigroup $S$ is an inverse semigroup if and only if every element admits a generalized inverse and idempotent elements in $S$ commute;
see \cite[Theorem~V.1.2]{How76IntroSgpThy}.\end{rmk}

Given an approximately unital \ba{} $A$, note that a contractive element in $A$ is a partial isometry if and only if it admits a generalized inverse in the multiplicative semigroup $A_1$.
This suggests that MP-partial isometries can be used to form inverse semigroups.
The next result describes when this happens.

\begin{prp}
\label{prp:Normalizer-InvSgp}
Let $A$ be an approximately unital \ba, and let $E$ be a family of pairwise commuting hermitian idempotents. 
Then the set of normalizers
\[
N_E := \big\{ a \in \PIMP(A) : a^\dag a \in E, aa^\dag \in E, a^\dag E a \subseteq E \text{ and } a E a^\dag \subseteq E \big\}
\]
is an inverse semigroup under multiplication.
\end{prp}
\begin{proof}
It is easy to verify that $N_E$ is closed under multiplication and under Moore-Penrose inversion.
Thus, $N_E$ is a semigroup with generalized inverses.
By \cref{rmk:CharInvSmgp}, it suffices to show that idempotents in $N_E$ commute.
For this, it suffices to show that the idempotents of $N_E$ belong to $E$. Let $e\in N_E$ be an idempotent.
\medskip

First, we show that $e^\dag=e$.
By construction, $e^\dag e$ and $ee^\dag$ belong to $E$ and therefore commute.
We have $(e^\dag)^2 = e^\dag$, and thus
\[
e^\dag 
= e^\dag e e^\dag
= \big( e^\dag e \big) \big( e e^\dag \big)
= \big( e e^\dag \big) \big( e^\dag e \big)
= e e^\dag e
= e.
\]

Now, the normalization condition that defines $N_E$ gives, in particular that $e^\dag e=e^2=e$ belongs to $E$, as desired.
\end{proof}

\begin{thm}
\label{prp:CharPIMP-InvSgp}
Let $A$ be an approximately unital \ba.
Then $\PIMP(A)$ is an inverse semigroup under multiplication if and only if hermitian idempotents in $A$ commute and $a^\dag ea$ is a hermitian idempotent for every $a \in \PIMP(A)$ and every hermitian idempotent~$e$.
\end{thm}
\begin{proof}
The backward implication holds by \cref{prp:Normalizer-InvSgp}.
Conversely, assume that $\PIMP(A)$ is an inverse semigroup under multiplication.
Since hermitian idempotents in $A$ are idempotent elements in $\PIMP(A)$, and since idempotent elements in an inverse semigroup commute, we get that hermitian idempotents in $A$ commute.
Further, given a hermitian idempotent $e\in A$ and $a \in \PIMP(A)$, it follows that $(a^\dag ea)^\dag = a^\dag ea = (a^\dag ea)^2$, which implies that $a^\dag ea$ is a hermitian idempotent.
\end{proof}

We now consider a class of hermitian idempotents for which the normalization condition in \cref{prp:CharPIMP-InvSgp} simplifies.
This will be used in the proof of \cref{prp:PIMP-InvSgp}.

Let $e$ be an idempotent in a unital \ba{} $A$, and set $\bar{e}:=1-e$, which is also an idempotent.
For $t \in \RR$, one computes that
\[
\exp(ite) = \bar{e} + \exp(it) e.
\]
It follows that $e$ is hermitian if and only if $\| \bar{e} + \beta e \| = 1$ for every $\beta \in \TT$.
Since multiplication by an element in $\TT$ is an isometry, it follows that $e$ is hermitian if and only if $\| \alpha \bar{e} + \beta e \| = 1$ for every $\alpha,\beta \in \TT$. 
(One says that $e$ is `bicircular'.)
Using this, one can show that $e$ is hermitian if and only if $\| \alpha \bar{e} + \beta e \| \leq 1$ for every $\alpha,\beta \in \CC$ with $|\alpha|,|\beta|\leq 1$.
Spain \cite{Spa17NumRangeSimpleCompression} introduced a strengthening of this concept:

\begin{dfn}
\label{pgr:Ultrahermitian}
An idempotent $e$ in a unital \ba{} $A$ is \emph{ultrahermitian} if $\big\| eae+(1-e)b(1-e)\big\| \leq 1$ for every $a,b \in A_1$.

We say that an element in an approximately unital \ba{} is an \emph{ultrahermitian idempotent} if its image in the minimal multiplier unitization is.
\end{dfn}

Ultrahermitian idempotents are hermitian, and Spain \cite{Spa17NumRangeSimpleCompression} gives an example of a hermitian idempotent that is not ultrahermitian. 
For idempotents on $L^p$-spaces, we show next that the two notions agree:

\begin{lma}
\label{lma:HerIdemLpUltra}
Let $p\in [1,\infty)$ and let $E$ be an $L^p$-space. 
Then every hermitian idempotent in $\mathcal{B}(E)$ is ultrahermitian.
\end{lma}
\begin{proof} 
Let $e\in \mathcal{B}(E)$ be a hermitian idempotent and let $a,b\in \mathcal{B}(E)$ be contractions. 
Use \cite[Proposition~2.6]{ChoGarThi24LpRigidity} to find a localizable measure $(X,\mu)$ such that $L^p(X,\mu)$ is isometrically isomorphic to $E$; we will identify these two Banach spaces from now on. 
By \cite[Proposition~2.7]{ChoGarThi24LpRigidity}, there exists a measurable subset $Y\subseteq X$ such that $e$ is the multiplication operator associated to the indicator function of $Y$. 
Let $\xi\in L^p(X,\mu)$. 
Then $eae(\xi)$ and $\bar{e}b\bar{e}(\xi)$ are supported in $Y$ and $Y^c:=X\setminus Y$, respectively.
Identifying $L^p(X,\mu)$ with $L^p(Y,\mu_Y)\oplus_p L^p(Y^c, \mu_{Y^c})$, it follows that $eae+\bar{e}b\bar{e}$ is a diagonal operator. 
Thus
\[
\|eae+\bar{e}b\bar{e}\|
= \max\big\{ \|eae\|, \|\bar{e}b\bar{e}\| \big\}
\leq 1.\qedhere
\]
\end{proof}

\begin{thm}
\label{prp:PIMP-InvSgp}
Let $A$ be an approximately unital \ba{} in which hermitian idempotents are ultrahermitian and commute.
Then $\PIMP(A)$ is an inverse semigroup under multiplication.
\end{thm}
\begin{proof}
To verify the conditions of \cref{prp:CharPIMP-InvSgp}, let $a \in \PIMP(A)$ and let $e \in A$ be a hermitian idempotent.
We need to shown that $a^\dag ea$ is hermitian.
By working in the minimal multiplier unitization of $A$, we may assume that $A$ is unital.
Set $f := a^\dag a$ and $\bar{f}:=1-f$.
Since $aa^\dag$ commutes with $e$, the element $a^\dag e a$ is an idempotent.
Given $t \in \RR$, it follows that
\[
\exp(it a^\dag e a)
= (1-a^\dag a) + a^\dag \exp(it e) a
= \bar{f} + f\big(a^\dag \exp(it e) a\big)f.
\]
Since $e$ is hermitian, it follows that $\exp(ite)$ is contractive.
By definition, $a^\dag$ and $a$ are contractive as well, and hence so is $a^\dag \exp(it e) a$.
Since $f$ is an ultrahermitian idempotent, it follows that $\exp(it a^\dag e a)$ is contractive, which implies that $a^\dag e a$ is hermitian.
Since hermitian idempotents in $A$ commute, we can now apply \cref{prp:CharPIMP-InvSgp} to deduce that $\PIMP(A)$ is an inverse semigroup under multiplication.
\end{proof}

\begin{rmk}
Let $A$ be an approximately unital \ba{} whose hermitian elements are closed under multiplication, and assume that the spectrum of the $\rm C^*$-core is connected.
Then $\PIMP(A) = \{0\}$ if $A$ is non-unital.
If $A$ is unital, then $\PIMP(A)$ consists of $0$ together with the group of invertible isometries in $A$. 
\end{rmk}

\begin{pgr}[$L^p$-operator algebras]
\label{pgr:LPOA}
Let $p \in [1,\infty)$.
A \ba{} $A$ is an \emph{$L^p$-operator algebra} if there exists an isometric representation of $A$ on some $L^p$-space.
If $A$ is an approximately unital \ba{} that admits a nondegenerate, isometric representation $\varphi\colon A \to \Bdd(E)$ on some $L^p$-space $E$, then it follows from \cite[Lemma~2.24]{BlePhi19LpOpAlgApproxId1} that $\varphi$ extends to a unital, isometric representation of $\widetilde{A} \to \Bdd(E)$, and in particular $\widetilde{A}$ is an \lpoa{} as well, and $A$ is \emph{unitizable} in the sense of \cite{GarTin25arX:CKUniquLpGraph}.

If $A$ is an approximately unital \lpoa{}, and $p \neq 1$, then $A$ admits a nondegenerate isometric representation on some $L^p$-space by \cite[Theorem~4.3]{GarThi20ExtendingRepr}, and it follows that $\widetilde{A}$ is an \lpoa{}.
For $p=1$ this is not clear.
\end{pgr}

The following lemma is straightforward, and we record it here, without proof, for future use. 

\begin{lma}
\label{prp:MapHermIdem}
Let $\varphi \colon A \to B$ be an approximately unital homomorphism between approximately unital \ba{s}.
Then:
\begin{enumerate}
\item
If $\varphi$ is contractive and $e$ is a hermitian (ultrahermitian) idempotent in $A$, then $\varphi(e)$ is a hermitian (ultrahermitian) idempotent in $B$.
\item
If $\varphi$ is isometric, then $e \in A$ is a hermitian (ultrahermitian) idempotent in $A$ whenever $\varphi(e)$ is a hermitian (ultrahermitian) idempotent in $B$.
\end{enumerate}
\end{lma}
%

\begin{prp}
\label{prp:PIMP-InvSgp-LPOA}
Let $p \in [1,\infty) \setminus \{2\}$, and let $A$ be an approximately unital \lpoa.
If $p=1$, additionally assume that $A$ admits a nondegenerate, isometric representation on some $L^1$-space (this is automatic if $A$ is unital).

Then hermitian idempotents in $A$ are ultrahermitian and commute.
Further, $\PIMP(A)$ is an inverse semigroup under multiplication.
In particular, the product of hermitian idempotents is again a hermitian idempotent.
\end{prp}
\begin{proof}
As explained in \cref{pgr:LPOA}, the assumptions imply that $\widetilde{A}$ admits a unital, isometric representation $\varphi \colon \widetilde{A} \to \Bdd(E)$ on some $L^p$-space $E$.
Let $e \in A$ be a hermitian idempotent.
Then $\varphi(e)$ is a hermitian idempotent, and since hermitian idempotents on $L^p$-spaces are ultrahermitian by \cref{lma:HerIdemLpUltra}, we first get that $\varphi(e)$ is ultrahermitian, and hence so is $e$ by part~(2) of~\cref{prp:MapHermIdem}.

Further, it is well-known that hermitian operators on $E$ are multiplication operators (using that $p \neq 2$) and therefore commute;
see, for example, \cite[Corollary~2.8]{ChoGarThi24LpRigidity}.
It follows that hermitian elements in $A$ commute as well.
Now \cref{prp:PIMP-InvSgp} shows that $\PIMP(A)$ is an inverse semigroup.
\end{proof}

We now address the question of norm continuity of the Moore-Penrose inverse, and at the same time establish Rako\v{c}evi{\'c}'s conjecture from \cite{Rak88MPInvBAlg} for certain \ba{s} including all \lpoa{s}.

\begin{con}[Rako\v{c}evi{\'c}'s conjecture]
\label{cnj}
\label{con:RakocvevicsConjecture}
Let $A$ be a unital \ba, and let $(a_n)_{n\in\NN}$ be a sequence in $A_\subMP$ that converges to $a\in A_\subMP$.
Then $a_n^\dag \to a^\dag$ if and only if $\sup_{n\in\NN} \| a_n^\dag \| < \infty$.
\end{con}

\begin{thm} 
\label{prp:ContinuityMP}
Let $A$ be an approximately unital \ba{} whose hermitian idempotents are ultrahermitian and commute.
Let $(a_n)_{n\in\NN}$ be a sequence in $A_\subMP$ that converges to $a\in A_\subMP$.
Then the following are equivalent:
\begin{enumerate}
\item
We have $a_n^\dag \to a^\dag$.
\item
We have $\sup_{n\in\NN} \| a_n^\dag \| < \infty$.
\item
We have $a_n^\dag a_n \to a^\dag a$ and $a_na_n^\dag \to aa^\dag$.
\end{enumerate}
\end{thm}
\begin{proof}
By passing to the minimal multiplier unitization, we may assume that $A$ is unital.
Let $E$ denote the set of hermitian idempotents in $A$, and let $D$ denote the closed subalgebra of $A$ generated by $E$.
Since the product of commuting ultrahermitian idempotents is (ultra)hermitian, if follows that $E$ is closed under multiplication. 
Using also that $A_\subHerm + i A_\subHerm$ is a closed subspace, we have
\[
D
= \overline{\linSpan(E)}^{\|\cdot\|} \subseteq A_\subHerm + i A_\subHerm.
\]
We have $D_\subHerm = D \cap A_\subHerm$, and it follows that $D$ is a unital \ba{} satisfying $D=D_\subHerm + i D_\subHerm$.
By the Vidav-Palmer theorem, $D$ is a unital, commutative \ca{}.
We let $X$ denote its spectrum, so that $D \cong C(X)$.

We will use that hermitian idempotents in $A$ (equivalently, in~$D$) are nothing but projections in $C(X)$.
We further use the standard fact that if $e$ and $f$ are two projections in a commutative \ca{} such that $\| e -f \| < 1$, then $e = f$.

\medskip

It is clear that~(1) implies~(2).
Supposing that~(2) holds, let us verify~(3).
Set
\[
N := \max \Big\lbrace  \| a^\dag \| , \sup_{n\in\NN} \| a_{n}^{\dag} \| \Big\rbrace.
\]
We introduce auxiliary maps $r,s \colon A_\subMP \to A$ given by 
\[
r(b) = bb^\dag, \andSep
s(b) = b^\dag b
\]
for $b \in A_\subMP$.
Note that~$r(b)$ and~$s(b)$ are hermitian idempotents and $b = r(b) b = b s(b)$ for every $b \in A_\subMP$.
We need to show that $r(a_n) \to r(a)$ and $s(a_n) \to s(a)$.
In fact, we will show that $r(a_n) = r(a)$ and $s(a_n) = s(a)$ for all $n$ large enough.
If $n$ satisfies $\| a_{n} - a \| < \tfrac{1}{2N}$, then 
\[
\big\| a_{n} - a_{n} s(a) \big\| 
= \big\| a_{n} - a + as(a) - a_{n} s(a) \big\| 
\leq \big\| a_{n} - a \big\| + \big\| a s(a) - a_{n} s(a) \big\| 
< \frac{1}{N},
\]
and hence 
\[
\big\| s(a_{n}) - s(a_{n}) s(a) \big\| 
= \big\| a_n^\dag a_{n} - a_n^\dag a_n s(a) \big\| 
\leq N \big\| a_{n} - a_{n} s(a) \big\| < 1.
\]
Since $s(a_{n})$ and $s(a_{n}) s(a)$ are projections in $C(X)$, we get $s(a_{n}) = s(a_{n}) s(a)$.
Reversing the roles of $a$ and $a_{n}$, we obtain also that $s(a) = s(a)s(a_n)=s(a_n)s(a)$, hence $s(a_{n}) = s(a)$.
Analogously, one shows that $r(a_{n}) = r(a)$.

\medskip

Assuming that~(3) holds, let us verify~(1).
By assumption, $(a_{n}^{\dag} a_n)_{n\in\NN}$ is a sequence of projections in $C(X)$ converging to the projection~$a^\dag a$.
This implies that $a_{n}^{\dag} a_n = a^\dag a$ for large enough $n$, and by discarding finitely many terms, we may assume that $a_{n}^{\dag} a_n = a^\dag a$ for all $n\in\NN$.
Analogously, we may assume that $a_n a_{n}^{\dag} = a a^\dag$ for all $n\in\NN$.

Set $e := a^\dag a$, and consider the Banach algebra $eAe \subseteq A$, which has $e$ as a unit.
Given $n\in\NN$, we have 
\[
a^{\dag} a_n 
= a^{\dag} a a^{\dag} a_n a_n^{\dag} a_n
= e a^{\dag} a_n e, \andSep
a_n^{\dag} a
= a_n^{\dag} a_n a_n^{\dag} a a^{\dag} a
= e a_n^{\dag} a e.
\]
Thus, $a^{\dag} a_n$ and $a_n^\dag a$ belong to $eAe$, and we have
\begin{align*}
\big( a^{\dag} a_n \big) \big( a_{n}^{\dag} a \big)
&= a^{\dag} \big( a_n a_{n}^{\dag} \big) a 
= a^{\dag} \big( a a^{\dag} \big) a 
= e, \andSep \\
\big( a_{n}^{\dag} a \big) \big( a^{\dag} a_n \big) 
&= a_n^{\dag} \big( a a^{\dag} \big) a_n 
= a_n^{\dag} \big( a_n a_n^{\dag} \big) a_n 
= e, 
\end{align*}
Thus, $a^{\dag} a_n$ is invertible in $eAe$ with inverse $a_{n}^{\dag} a$, for each~$n$.
Since left and right multiplication by a fixed element in $A$ are continuous, and since inversion is continuous in $\mathrm{GL}(eAe)$, we obtain the following chain of implications:
\begin{align*}
a_n \xrightarrow{\|\cdot\|} a 
&\ \implies \ a^\dag a_n \xrightarrow{\|\cdot\|} a^\dag a = e 
\ \implies \ a_{n}^{\dag} a \xrightarrow{\|\cdot\|} e \\
&\ \implies \ a_{n}^{\dag} = a_{n}^{\dag} a_{n} a_{n}^{\dag} = a_{n}^{\dag} a a^\dag \xrightarrow{\|\cdot\|} e a^\dag = a^\dag.
\end{align*}
This verifies~(1), as desired.
\end{proof}

\begin{rmk}
We note that the proof of \cref{prp:ContinuityMP} only uses that the family of hermitian idempotents is closed under multiplication.
This holds under the assumption of the theorem, since the product of two commuting ultrahermitian idempotents is ultrahermitian.
On the other hand, it is not clear if the product of two commuting hermitian idempotents is hermitian.
\end{rmk}

\begin{thm} 
Rako\v{c}evi{\'c}'s conjecture holds for unital \lpoa{s}, for $p \in [1,\infty)$.
\end{thm}
\begin{proof}
For $p \in [1,\infty) \setminus \{2\}$, this follows from \cref{prp:ContinuityMP} since hermitian idempotents in unital \lpoa{s} commute and are closed under multiplication by \cref{prp:PIMP-InvSgp-LPOA}.

For $p=2$, let $A$ be a unital $L^2$-algebra and let $\varphi \colon A \to \Bdd(H)$ be a unital, isometric representation on some Hilbert space~$H$. 
Since $\varphi$ preserves hermitian idempotents, it follows that $\varphi(A_\subMP) \subseteq \Bdd(H)_\subMP$ and $\varphi(a^\dag)=\varphi(a)^\dag$ for all $a \in A_\subMP$.

Let $(a_n)_{n\in\NN}$ be a sequence in $A_\subMP$ converging to $a \in A_\subMP$.
If $a_{n}^{\dag} \to a^{\dag}$, then clearly $\sup_{n\in\NN} \| a_{n}^{\dag} \| < \infty$.
Conversely, suppose that $\sup_{n\in\NN} \| a_{n}^{\dag} \| < \infty$.
In \cite[Theorem~2.2]{Rak93ContinuityMPInvCAlg}, Rako\v{c}evi{\'c} confirmed his conjecture for \ca{s}.
Using that~$\varphi$ is an isometric representation and that $\Bdd(H)$ is a \ca{}, it follows that
\[
\varphi(a_{n}^{\dag}) 
= \varphi(a_{n})^\dag 
\to \varphi(a)^\dag 
= \varphi(a^\dag).
\]
Using again that $\varphi$ is isometric, we get $a_{n}^{\dag} \to a^\dag$. 
\end{proof}

We say that two MP-partial isometries $a_0$ and $a_1$ in an approximately unital \ba{} $A$ are \emph{homotopic}, denoted $a_0 \simeq a_1$, provided that there is a norm-continuous path $a\colon [0,1] \to \PIMP(A)$ with $a(0)=a_0$ and $a(1)=a_1$. 

\begin{lma} 
\label{prp:HtpyMultDagger}
Let $A$ be an approximately unital \ba{} in which hermitian idempotents are ultrahermitian and commute.
Then the following statements hold:
\begin{enumerate}
\item
If $a_0 \simeq a_1$ and $b_0 \simeq b_1$ in $\PIMP(A)$, then $a_0b_0 \simeq a_1b_1$.
\item
If $a_0 \simeq a_1$ in $\PIMP(A)$, then $a_0^\dag \simeq a_1^\dag$.
\end{enumerate}
\end{lma}
\begin{proof}
(1) 
If $t \mapsto a_t$ and $t \mapsto b_t$ are homotopies implementing $a_0 \simeq a_1$ and $b_0 \simeq b_1$, then by joint continuity of multiplication in $A$, and since $\PIMP(A)$ is a semigroup under multiplication by \cref{prp:PIMP-InvSgp}, we have that $t \mapsto a_t b_t$ is a continuous path in $\PIMP(A)$ from $a_0b_0$ to $a_1b_1$. 

(2)
Since elements in $\PIMP(A)$ are contractive, it follows from \cref{prp:ContinuityMP} that the dagger map $\dag$ is continuous on $\PIMP(A)$. 
Thus, if $t \mapsto a_t$ is a continuous path in $\PIMP(A)$ from $a_0$ to $a_1$, then $t \mapsto a_t^{\dag}$ is a continuous path from $a_0^\dag$ to $a_1^\dag$.
\end{proof}

\begin{dfn}
\label{dfn:SInv}
Let $A$ be an approximately unital \ba.
We denote the set of homotopy classes of its MP-partial isometries by
\[
\SInv(A) := \PIMP(A)/\simeq,
\]
and we write $[a]$ for the class of $a \in \PIMP(A)$ in $\SInv(A)$.
\end{dfn}

\begin{thm} 
\label{prp:InvSgpHtpyPIMP}
Let $A$ be an approximately unital \ba{} in which hermitian idempotents are ultrahermitian and commute.
Then $\SInv(A)$ is an inverse semigroup with multiplication given by $[a] [b] := [ab]$ for $a,b \in \PIMP(A)$. In this inverse semigroup, we have 
$[a]^\dag := [a^\dag]$ for all $a\in \PIMP(A)$. Moreover, 
the quotient map $\PIMP(A) \to \SInv(A)$ is a homomorphism of inverse semigroups.

This applies to approximately unital \lpoa{s} for $p \in (1,\infty) \setminus \{2\}$, as well as for approximately unital $L^1$-operator algebras that admit a nondegenerate, isometric representation on some $L^1$-space.
\end{thm}
\begin{proof}
It follows from \cref{prp:HtpyMultDagger} that the multiplication and the $\dag$-operation in~$\SInv(A)$ are well defined.
Given $[a] \in \SInv(A)$, using that $\PIMP(A)$ is an inverse semigroup by \cref{prp:PIMP-InvSgp} it follows that
\[
[a] 
= [a a^\dag a]
= [a] [a^\dag] [a], \andSep
[a^\dag] 
= [a^\dag a a^\dag]
= [a^\dag] [a] [a^\dag].
\]
We need to show that $[a^\dag]$ is the unique element in $\SInv(A)$ with this property.
Suppose that $b \in \PIMP(A)$ satisfies 
\begin{align*}\label{eq1}\tag{2.1}a \simeq a b a \andSep b \simeq bab; \end{align*}
we want to show that $a^\dag \simeq b$. 
Applying part~(2) of~\cref{prp:HtpyMultDagger} to $b \simeq bab$, we get 
\begin{align*}\label{eq2}\tag{2.2}
b^\dag \simeq b^\dag a^\dag b^\dag.
\end{align*}
Thus,
\begin{align*}\label{eq3}\tag{2.3}
a^\dag = a^\dag a a^\dag  \stackrel{\eqref{eq1}}{\simeq} (a^\dag a) b (a a^\dag) \andSep 
b^\dag = b^\dag b b^\dag  \stackrel{\eqref{eq1}}{\simeq} (b^\dag b) a (b b^\dag).
\end{align*}
Using part~(1) of~\cref{prp:HtpyMultDagger} repeatedly, using $a^\dag=a^\dag a a^\dag$ at the third step, and using that hermitian idempotents commute in the fifth step, we get
\begin{align*}
b =bb^\dag b &\stackrel{\eqref{eq2}}{\simeq} b (b^\dag a^\dag b^\dag) b \\
&\ = \ (b b^\dag) (a^\dag a a^\dag) (b^\dag b) \\
&\stackrel{\eqref{eq1}}{\simeq} 
(b b^\dag) (a^\dag aba a^\dag) (b^\dag b) \\
& \ = \ (a^\dag a) (b b^\dag) b (b^\dag b) (a a^\dag) \\
&\ = \ (a^\dag a) b (a a^\dag) \stackrel{\eqref{eq3}}{\simeq} a^\dag,
\end{align*}
as desired. We have seen that $\SInv(A)$ is an inverse semigroup, and that the quotient map $\PIMP(A) \to \SInv(A)$ is a homomorphism of inverse semigroups.

Finally, by \cref{prp:PIMP-InvSgp-LPOA} the assumptions are satisfied for the mentioned classes of \lpoa{s}.
\end{proof}

We point out that the related notions of \emph{$L^p$-projections} and \emph{$L^p$-partial isometries} have been studied by Bardadyn, Kwa{\'s}niewski and McKee
in 
Section~2 of~\cite{BarKwaMck23arX:BAlgTwGpdInvSgpDisintLp}.

\section{Partial isometries in groupoid \texorpdfstring{$L^p$}{Lp}-operator algebras} 
\label{sec:PI-FpGpd}

Given a Hausdorff, \etale{} groupoid $\calG$, we consider the associated reduced \lpoa{} $\Fp(\calG)$ (\cref{pgr:GpdLPOA}), which is an approximately unital \ba{} that admits a nondegenerate, isometric representation on an $L^p$-space.
For $p \neq 2$, we will show that the set $\PIMP(\Fp(\calG))$ of its MP-partial isometries is naturally an inverse semigroup (\cref{prp:PIMP-InvSgp-LPOA}), and the set $\SInv(\Fp(\calG))$ of path connected components in $\PIMP(\Fp(\calG))$ inherits the structure of an inverse semigroup (\cref{prp:InvSgpHtpyPIMP}).
These results follow from our computations of these inverse semigroups in terms of the groupoid~$\calG$.
Indeed, the main result of this section shows that the inverse semigroup $\SInv(\Fp(\calG))$ is naturally isomorphic to the inverse semigroup of compact, open bisections in $\calG$;
see \cref{prp:Bco_SFp}.

\medskip

We start by recalling basic definitions and results.
For details and more information, the reader is referred to \cite{Ren80GrpdApproach, Sim17arX:EtaleGpds, Wil19ToolKitGpdCAlgs} for the theory of groupoids, and to  \cite{GarLup17ReprGrpdLp} and \cite[Section~4]{ChoGarThi24LpRigidity} for the theory of groupoid \lpoa{s}.

\begin{pgr}[Groupoids]
\label{pgr:groupoids}
Let $\calG$ be a groupoid with unit space $\calG^{(0)}$ and let 
$\ran, \sou\colon \calG\to\calG^{(0)}$ denote the range and source maps, respectively, which are given by $\ran(x)= xx^{-1}$ and $\sou(x)= x^{-1} x$ for 
all $x\in\calG$. 
Given $x,y \in \calG$, the product $xy$ is defined if (and only if) $\sou(x)=\ran(y)$.
An \emph{\'etale groupoid} is a groupoid equipped with a (not necessarily Hausdorff) locally compact topology such that the multiplication and inversion maps are continuous, such that $\calG^{(0)}$ is Hausdorff in the relative topology, and such that the range map (and therefore also the source map) is a local homeomorphism;
see, for example \cite[Definition~2.4.1]{Sim17arX:EtaleGpds}.
We warn the reader that in parts of the literature an \'etale groupoid is assumed to be Hausdorff, while in others it is not assumed to be locally compact.
In an \'etale groupoid $\calG$, the unit space $\calG^{(0)}$ is always an open subset.
However, $\calG^{(0)}$ is closed (hence clopen) if and only if $\calG$ is Hausdorff;
see \cite[Lemmas~2.3.2 and~2.4.2]{Sim17arX:EtaleGpds}.

A \emph{bisection} in a groupoid $\calG$ is a subset $U \subseteq \calG$ for which $\sou |_U$ and $\ran |_U$ are injective.
For a groupoid $\calG$, we let~$\Bo(\calG)$ denote the family of all open bisections, and we let $\Bco(\calG)$ denote the family of all compact, open bisections. 
Both are inverse semigroups under set multiplication and set inversion, given by $$A B = \{ ab \colon \sou(a)=\ran(b) ; a \in A ; b \in B \} \andSep A^{-1} = \{ a^{-1} \colon a \in A \}.$$ 

It is well-known that an \'etale groupoid admits a basis for its topology consisting of open bisections.
An \emph{ample groupoid} is an \'etale groupoid that has a basis for its topology consisting of compact, open bisections.
An \'etale groupoid is ample if and only if its unit space is totally disconnected. 
Using this, it is easy to see that a compact, open bisection in an ample groupoid 
is totaly disconnected in the relative topology.

Given $X \subseteq \calG^{(0)}$, we use the notation 
\[
\calG_X = \big\{ x \in \calG : \sou(x) \in X \big\}, \andSep
\calG^{X} = \big\{ x \in \calG : \ran(x) \in X \big\}.
\]
We shall write $\calG_u$ and $\calG^{u}$ instead of $\calG_{\{u\}}$ and $\calG^{\{u\}}$ whenever $u \in \calG^{(0)}$ is a unit. 

A \emph{homomorphism} between two \'etale groupoids $\calG$ and $\calH$, is a continuous map $\varphi \colon \calG \to \calH$ such that if $(x,y)$ is a composable pair in $\calG$, then $(\varphi(x) , \varphi(y))$ is a composable pair in $\calH$, and in this case $\varphi(xy) = \varphi(x) \varphi(y)$. 
Two \'etale groupoids $\calG, \calH$ are said to be \emph{isomorphic}, in symbols $\calG \cong \calH$, if there is a groupoid homomorphism $\calG \to \calH$ which is also a homeomorphism.
\end{pgr}

\begin{pgr}[Convolution algebras]
\label{pgr:ConvAlgs}
Let $\calG$ be a Hausdorff, \etale{} groupoid, and let $C_c (\calG)$ denote the space of continuous, compactly supported functions $\calG \to \CC$. 
We endow~$C_c (\calG)$ with the convolution product, which for $f,g \in C_c (\calG)$ is given by \[
(f \ast g)(x) = \sum_{y \in \calG_{\sou(x)}} f(x y^{-1}) g(y) = \sum_{y \in \calG^{\ran(x)}} f(y) g(y^{-1}x),
\]
for $x \in \calG$.
Further, $C_c (\calG)$ is equipped with an involution defined by
\[
f^{\ast}(x) = \overline{f(x^{-1})}
\]
for $f \in C_c (\calG)$ and $x \in \calG$.
This turns $C_c (\calG)$ into an associative $\ast$-algebra.

The \emph{$I$-norm} on $C_c (\calG)$ is given by 
\[
\| f \|_I = \max \left\lbrace \sup_{u \in \calG^{(0)}} \sum_{x \in \calG_u} |f(x)| \, , \, \sup_{u \in \calG^{(0)}} \sum_{x \in \calG^{u}} |f(x)| \right\rbrace
\]
for $f \in C_c(\calG)$.
The $I$-norm is submultiplicative and satisfies $\|f^*\|_I=\|f\|_I$.
\end{pgr}

\begin{pgr}[Groupoid \lpoa{s}]
\label{pgr:GpdLPOA}
Let $\calG$ be a Hausdorff, \etale{} groupoid, and let $p \in [1, \infty)$.
Given a unit $u \in \calG^{(0)}$ and $f \in C_c(\calG)$, the operator $\lambda_u(f)$ on $\ell^p (\calG_u)$ is defined by 
\[
\lambda_u (f) (\xi) (x) = \sum_{y \in \calG_u} f(x y^{-1}) \xi (y),
\]
for $x \in \calG_{u}$ and $\xi \in \ell^p (\calG_u)$. 
The resulting map $\lambda_u \colon C_c (\calG) \to \mathcal{B}(\ell^p (\calG_u))$ is a contractive homomorphism. 
The \emph{reduced groupoid $L^p $-operator algebra} associated to~$\calG$, denoted $\Fp(\calG)$, is defined to be the completion of $C_c(\calG)$ under the norm
\[
\lVert f \rVert_{\lambda} := \sup_{u \in \calG^{(0)}} \lVert \lambda_{u}(f) \rVert.
\]

The Banach algebra $\Fp(\calG)$ is approximately unital, with any contractive approximate unit for $C_c(\calG^{(0)})$ being also an approximate unit for $\Fp(\calG)$.
\end{pgr}

It is easy to see that $\Fp(\calG)$ is an \lpoa, as for example the representation $\bigoplus_{u\in\calG^{(0)}}\lambda_u$ on $\bigoplus_{u\in\calG^{(0)}} \ell^p(\mathcal{G}_u)$ is isometric. 
A more convenient isometric representation
of $F^p_\lambda(\calG)$ on an $L^p$-space is constructed as follows.

\begin{pgr}[The left regular representation]
\label{pgr:lambda}
Let $\calG$ be a Hausdorff, \etale{} groupoid, and let $p \in [1, \infty)$.
Denote by $\CC\calG$ the (dense) subspace of $\ell^p(\calG)$ consisting of those functions $\xi\colon \calG\to\CC$ with finite support.
Given $f \in C_c(\calG)$, it is easy to see that 
there is a unique operator $\lambda(f) \in \Bdd(\ell^p(\calG))$ satisfying
\[
\lambda (f) (\xi) = f \ast \xi
\]
for all $\xi \in \CC \calG$. 
We obtain a contractive homomorphism $\lambda \colon C_c (\calG) \to \mathcal{B}(\ell^p (\calG))$.
For each $u \in \calG^{(0)}$, the subspace $\ell^p (\calG_u) \subseteq \ell^p(\calG)$ is invariant under the operators~$\lambda(f)$ for $f \in C_c (\calG)$.
The map $\Psi \colon \bigoplus_{u \in \calG^{(0)}} \ell^p (\calG_u) \to \ell^p (\calG)$ that is the extension of the map $\Psi ((\xi_u)_{u \in F} ) = \sum_{u \in F} \xi_u$, for $\xi_u \in \CC \calG_u$ and $F \subseteq \calG^{(0)}$ finite, is an isometric isomorphism with inverse given by the extension of the map $\Psi^{-1} (\xi) = (\xi_u)_{u \in F} $, where $\xi_u = \xi \cdot 1_{\calG_u}$ and $F = \sou(\supp (\xi))$, for $\xi \in \CC \calG$. 
It follows that 
\[
\lambda(f) = \Psi \circ \bigoplus_{u \in \calG^{(0)}} \lambda_u (f) \circ \Psi^{-1}
\]
for each $f \in C_c (\calG)$. 
Thus, $\| f \|_\lambda = \| \lambda(f) \|$ for each $f \in C_c (\calG)$ and so $\lambda$ extends to an isometric representation $\lambda \colon \Fp(\calG) \to \mathcal{B}(\ell^p (\calG))$, called the \emph{left regular representation} of $\calG$, through which $\Fp(\calG)$ acts non-degenerately on $\ell^p (\calG)$.
\end{pgr}

The following norm estimates will be used repeatedly.

\begin{rmk}
\label{rmk:NormFunctionBisection}
Let $\calG$ be a Hausdorff, \etale{} groupoid, and let $p \in [1, \infty)$.
For any $f \in C_c (\calG)$, we have
\[
\lVert f \rVert_{\infty} \leq \lVert f \rVert_{\lambda} \leq \lVert f \rVert_I.
\]
If $f \in C_c (\calG)$ is supported on a bisection, then $\| f \|_I \leq \|f\|_\infty$ and it follows that
\[
\lVert f \rVert_{\infty} = \lVert f \rVert_{\lambda} = \lVert f \rVert_I.
\]
\end{rmk}

\begin{pgr}[Renault's $j$-map]\label{pgr:Renault}
Let $\calG$ be a Hausdorff, \etale{} groupoid, and let $p \in [1, \infty)$.
The map $j \colon F_{\lambda}^{p}(\calG) \to C_0 (\calG)$ given by 
\[
j(a) (x) = \lambda (a) (\delta_{\sou(x)}) (x) = \lambda_{\sou(x)} (a) (\delta_{\sou(x)}) (x) ,
\]
for $a \in F_{\lambda}^p (\calG)$ and $x \in \calG$, is contractive, linear and injective, and extends the identity on $C_c (\calG)$. 
We shall refer to this map as \emph{Renault's $j$-map}. 
Given $a,b \in \Fp(\calG)$ and $x \in \calG$, we have that $j (a b) (x) = \big(j (a) \ast j (b)\big) (x)$, where the sum defining $\big(j (a) \ast j (b)\big) (x)$ is absolutely convergent;
see \cite[Section~4]{ChoGarThi24LpRigidity}.
\end{pgr}

The $\rm C^*$-core and hermitian elements of a reduced groupoid $L^p$-operator algebra associated with a Hausdorff, \etale{} groupoid are described in the next lemma, which extends \cite[Proposition~5.1]{ChoGarThi24LpRigidity} to the case of non-compact unit spaces.

\begin{lma} 
\label{prp:CoreFpGpd}
Let $\calG$ be a Hausdorff, \'etale groupoid, and let $p \in [1, \infty) \setminus \{2\}$.
Using that $\calG^{(0)} \subseteq \calG$ is clopen, we obtain a natural restriction map
$\mathrm{res} \colon C_0(\calG) \to C_0(\calG^{(0)})$.
Then $\mathrm{res}\circ j$ induces natural isomorphisms
\[
\core(\Fp(\calG)) \cong C_0 (\calG^{(0)}) \andSep
\Fp(\calG)_{\rm h} \cong C_0 (\calG^{(0)}, \RR).
\]
\end{lma}
\begin{proof}
The isometric representation $\lambda \colon \Fp(\calG) \to \mathcal{B}(\ell^p (\calG))$ extends to a unital isometric representation
\[
\widetilde{\lambda} \colon \widetilde{\Fp(\calG)} \to \lambda(\Fp(\calG)) + \CC 1_{\ell^p (\calG)} \subseteq \mathcal{B}(\ell^p (\calG))
\]
because $\Fp(\calG)$ has a contractive approximate identity and acts non-degenerately on $\ell^p (\calG)$ via $\lambda$. 
Given an element $a \in \mathrm{core}(\Fp(\calG))$, we have
\[
\widetilde{\lambda} (a) = \lambda(a) \in \mathrm{core}(\mathcal{B}(\ell^p (\calG))).
\]
By \cite[Proposition~2.7]{ChoGarThi24LpRigidity}, there exists a bounded function $f \colon \calG \to \CC$ such that $\lambda (a)$ is the multiplication operator corresponding to $f$.
It follows that
\[
j(a) (x) = f(x) \delta_{\sou(x)} (x)
\]
for each $x \in \calG$, and consequently $j(a)(x) = 0$ for all $x \notin \calG^{(0)}$.
Using that Renault's $j$-map is injective and that $j(ab) = j(a) \ast j(b)$ for $a,b \in \Fp(\calG)$, it follows that $\mathrm{res} \circ j$ induces an injective homomorphism
\[
\core(\Fp(\calG)) \to C_0 (\calG^{(0)}).
\]

Next, using that $\calG^{(0)}$ is a clopen subset of $\calG$, we obtain a natural inclusion $\iota \colon C_c(\calG^{(0)}) \subseteq C_c(\calG)$.
One verifies that $\mathrm{res} \circ j \circ \lambda \circ \iota$ is the inclusion $C_c(\calG^{(0)}) \subseteq C_0(\calG^{(0)})$, which means that the following diagram commutes:
\[
\xymatrix@R-5pt{
C_c(\calG) \ar[r]^{\lambda} & \Fp(\calG) \ar[r]^{j} & C_0(\calG) \ar@{->>}[d]^{\mathrm{res}} \\
C_c(\calG^{(0)}) \ar@{^{(}->}[u]^{\iota} \ar@{^{(}->}[rr]  & & C_0(\calG^{(0)}).
}
\]
Since all maps are contractive, and since $C_c(\calG^{(0)}) \to C_0(\calG^{(0)})$ is isometric, it follows that $\lambda\circ\iota$ is isometric as well, and therefore extends to a map $C_0(\calG^{(0)}) \to \Fp(\calG)$.
We deduce that $\mathrm{res} \circ j$ induces an isometric, surjective homomorphism from $\core(\Fp(\calG))$ to $C_0 (\calG^{(0)})$, which moreover identifies $\Fp(\calG)_{\rm h}$ with $C_0 (\calG^{(0)}, \RR)$.
\end{proof}

Given $B \in \Bco(\calG)$, we let $\mathbbm{1}_B \in C_c (\calG)$ denote its associated indicator function.

\begin{lma}
\label{prp:BisectionPIMP}
Let $\calG$ be a Hausdorff, \'etale groupoid, and let $p \in [1, \infty)\setminus\{2\}$.
Let $B \in \Bco(\calG)$, and let $f \in C_c(\calG)$ be a function taking values in $\TT$ on $B$ and vanishing on the complement of $B$.
Then $f$ belongs to $\PIMP(\Fp(\calG))$, and we have $f^\dag = f^*$.
Further, $f^\dag \ast f = \mathbbm{1}_{\sou(B)}$ and $f \ast f^\dag = \mathbbm{1}_{\ran(B)}$.

The resulting map $\mathbbm{1}\colon \Bco(\calG) \to \PIMP(\Fp(\calG))$, given by $B\mapsto \mathbbm{1}_B$, is a homomorphism of inverse semigroups.
\end{lma}
\begin{proof}
Given $x \in \calG$, we have
\[
(f^* \ast f)(x) 
= \sum_{y \in \calG^{\ran(x)}} \overline{f(y^{-1})} f(y^{-1}x).
\]
Using that $f$ is supported on a bisection, we see that $(f^* \ast f)(x)$ can only be nonzero if $x \in \sou(B)$, and in this case it follows that $(f^* \ast f)(x)=1$ since $f(Bx) \in \TT$.
This shows that $f^* \ast f = \mathbbm{1}_{\sou(B)}$.
Similarly, we get $f \ast f^* = \mathbbm{1}_{\ran(B)}$.
Since $\mathbbm{1}_{\sou(B)}$ and $\mathbbm{1}_{\ran(B)}$ are hermitian idempotents in $\Fp(\calG)$, it follows that $f$ is Moore-Penrose invertible with $f^\dag = f^*$.
Since $f$ and $f^*$ are supported on bisections, by \cref{rmk:NormFunctionBisection} we have
\[
\| f \|_\lambda = \| f \|_\infty \leq 1, \andSep
\| f^* \|_\lambda = \| f^* \|_\infty \leq 1.
\]
showing that $f$ belongs to $\PIMP(\Fp(\calG))$.

In particular, for $B \in \Bco(\calG)$, we have $\mathbbm{1}_B \in \PIMP(\Fp(\calG))$ with $\mathbbm{1}_B^\dag = \mathbbm{1}_B^* = 1_{B^{-1}}$.
It is easy to see that $\mathbbm{1}_A \mathbbm{1}_B = 1_{AB}$ in $\Fp(\calG)$ for all $A,B \in \Bco(\calG)$, showing that the map $\Bco(\calG) \to \PIMP(\Fp(\calG))$ is a homomorphism of inverse semigroups. 
\end{proof}

\begin{cor}\label{cor:PhiBisections}
Let $\calG$ be a Hausdorff, \'etale groupoid, and let $p \in [1, \infty)\setminus\{2\}$.
Then there is a canonical homomorphism of inverse semigroups $\Phi_p\colon \Bco(\calG) \to \SInv(\Fp(\calG))$ given by $\Phi_p(B)= [\mathbbm{1}_B]$, for all $B\in \Bco(\calG)$.
\end{cor}
\begin{proof}
This follows immediately by composing 
the canonical map 
$\mathbbm{1}\colon \Bco(\calG) \to \PIMP(\Fp(\calG))$ from \cref{prp:BisectionPIMP} 
with the quotient map $\PIMP(\Fp(\calG)) \to \SInv(\Fp(\calG))$, which is also
a homomorphism of inverse semigroups by \cref{prp:InvSgpHtpyPIMP}.
\end{proof}

The rest of this section is dedicated to showing that, in the setting of Hausdorff, ample  groupoids, the map $\Phi_p$ from \cref{cor:PhiBisections} is an isomorphism of inverse semigroups. 

The proof of the following lemma would be easier if our groupoids were endowed with 
the discrete topology, so that indicator functions of singletons were elements in our groupoid algebras. 
Although this is not the case in general, we are still able to define
right convolution by those elements in a consistent way.

\begin{lma} 
\label{prp:RightConv}
Let $\calG$ be a Hausdorff, \etale{} groupoid, let $p \in [1, \infty)$, and let $y \in \calG$.
Then there is an operator $\rho(\mathbbm{1}_y) \in \Bdd(\ell^p(\calG))$ given by
\[
\rho(\mathbbm{1}_y)(\xi)(x)
= \begin{cases}
\xi\big(xy^{-1}\big), &\text{if $\sou(x)=\sou(y)$} \\
0, &\text{else}
\end{cases}
\]
for $\xi \in \CC\calG$ and $x \in \calG$.
Further, $\rho(\mathbbm{1}_y)$ commutes with each operator in the image of the left regular representation. 
\end{lma}
\begin{proof}
We show that $\rho(\mathbbm{1}_y) \colon \CC\calG \to \CC\calG$ defined by the formula in the statement is contractive with respect to the $\ell^p$-norms, and therefore extends to a contractive operator in $\Bdd(\ell^p(\calG))$.
Let $\xi \in \CC\calG$.
Viewing elements in $\CC\calG$ as tuples indexed over $\calG$, and using that $x_1y^{-1}=x_2y^{-1}$ implies $x_1=x_2$ for $x_1,x_2,y\in\calG$ with the same source, we see that the nonzero entries in the tuple $\rho(\mathbbm{1}_y)(\xi)$ arise, up to a permutation, also in the tuple~$\xi$, which immediately shows that $\|\rho(\mathbbm{1}_y)(\xi)\|_p \leq \|\xi\|_p$, as desired.

\medskip

Let $\lambda \colon \Fp(\calG) \to \mathcal{B}(\ell^p (\calG))$ denote the left regular representation, which we recall is given by $\lambda(f)(\xi) = f \ast \xi$, for $f \in C_c (\calG)$ and $\xi \in \CC \calG$; see \cref{pgr:lambda}.
We identify $\rho(\mathbbm{1}_y)$ with right convolution with $\mathbbm{1}_y$, the indicator function of the singleton ~$\{y\}$.
Thus, we identify $\lambda(f)\rho(\mathbbm{1}_y)\xi$ with $f\ast(\xi\ast \mathbbm{1}_y)$, and $\rho(\mathbbm{1}_y)\lambda(f)\xi$ with $(f\ast\xi)\ast \mathbbm{1}_y$.
Analogous to the proof of associativity of the convolution product, one then sees that the identity
\[\tag{3.1}\label{eq31}\lambda(f)\rho(\mathbbm{1}_y)\xi=\rho(\mathbbm{1}_y)\lambda(f)\xi\]
holds for all $f \in C_c(\calG)$ all $y \in \calG$, and all $\xi \in \CC\calG$, and then by continuity for all $\xi \in \ell^p(\calG)$.
Thus, $\rho(\mathbbm{1}_y)$ commutes with $\lambda(f)$ for each $f \in C_c(\calG)$.
By continuity, it follows that \eqref{eq31} holds
for all $a \in \Fp(\calG)$, as desired.
\end{proof}

Using Renault's $j$-map $j \colon F_{\lambda}^{p}(\calG) \to C_0 (\calG)$, we define the (open) support of an operator $a \in \Fp(\calG)$ as
\[
\supp(a) := \big\{ x \in \calG : j(a)(x) \neq 0 \big\}.
\]
The next result shows in particular that every MP-partial isometry in $\Fp(\calG)$ is contained in the dense subalgebra $C_c(\calG) \subseteq \Fp(\calG)$.

\begin{prp} 
\label{prp:StructurePIMP-FpGpd}
Let $\calG$ be a Hausdorff, \etale{} groupoid, let $p \in [1, \infty) \setminus \{2\}$, and let $a \in \PIMP(\Fp(\calG))$.
Then $B := \supp(a)$ is a compact, open bisection, and $a$ is $\TT$-valued on its support. 
Further, $\sou(B)=\supp(a^\dag a)$ and $\ran(B)=\supp(aa^\dag)$.
\end{prp}
\begin{proof}
Let $\lambda \colon \Fp(\calG) \to \mathcal{B}(\ell^p (\calG))$ denote the left regular representation.
Then $\lambda(a^\dag a)$ and $\lambda(aa^\dag)$ are hermitian idempotents in $\Bdd(\ell^p(\calG))$.
By \cite[Proposition~2.7]{ChoGarThi24LpRigidity}, there are subsets $E,F \subseteq \calG$ such that $\lambda(a^\dag a)$ is the multiplication  operator by the indicator function of~$E$, and similarly $\lambda(aa^\dag)$ is multiplication by the indicator function of~$F$. 

For $y \in \calG$, we consider the right convolution operator $\rho(\mathbbm{1}_y) \in \Bdd(\ell^p(\calG))$ as in \cref{prp:RightConv}.
If $\delta_x \in \ell^p(\calG)$ denotes the indicator function of $x \in \calG$, then
\[
\rho(\mathbbm{1}_y)(\delta_x) 
= \begin{cases}
\delta_{xy^{-1}}, &\text{if } \sou(x) = \sou(y) \\
0, &\text{otherwise}
\end{cases}.
\]

By \cref{prp:RightConv}, $\rho(\mathbbm{1}_y)$ commutes with $\lambda(a^\dag a)$.
In particular, for $x \in E$ and $y \in \calG$ with $\sou(x)=\ran(y)$, we get
\[
\lambda(a^\dag a)(\delta_{xy})
= \lambda(a^\dag a)\rho(\mathbbm{1}_{y^{-1}})(\delta_x)
= \rho(\mathbbm{1}_{y^{-1}})\lambda(a^\dag a)(\delta_x)
= \rho(\mathbbm{1}_{y^{-1}})(\delta_x)
= \delta_{xy},
\]
which shows that $xy \in E$.
Thus, for $U := \ran(E) \subseteq \calG^{(0)}$, we have $E=\calG^U$.
Similarly, for $V := \ran(F)$, we get $F=\calG^V$. Observe that
$U=\supp(a^\dag a)$ and $V=\supp(aa^\dag)$.

\medskip

Note that $\lambda(a)$ induces an invertible isometry from $\ell^p(\calG^U)$ to $\ell^p(\calG^V)$.
By Lamperti's theorem (see \cite[Theorem~2.12]{Gar21ModernLp}), there exist a function $f \colon \calG^V \to \TT$ and a bijection $\sigma \colon \calG^U \to \calG^V$ such that the corestriction $\lambda(a) \colon \ell^p(\calG^U) \to \ell^p(\calG^V)$ satisfies
\[
\lambda(a)(\xi)(x)
= f(x)\xi(\sigma^{-1}(x))
\]
for all $\xi \in \ell^p(\calG^V)$ and $x \in \calG^V$.

Define operators $\overline{m}_f, \overline{w}_\sigma$ on $\ell^p(\calG)$ by setting
\[
\overline{m}_f(\xi)(x) 
= \begin{cases}
f(x)\xi(x), &\text{if } x\in \calG^V \\
0, &\text{else}
\end{cases}, \andSep 
\overline{w}_\sigma(\xi)(x) 
= \begin{cases}
\xi(\sigma^{-1}(x)), &\text{if } x\in \calG^V \\
0, &\text{else}
\end{cases}
\]
for $\xi \in \ell^p(\calG)$ and $x \in \calG$.
Note that $\lambda(a) = \overline{m}_f \overline{w}_\sigma$.

\medskip

By \cref{prp:RightConv}, $\lambda(a)$ commutes with $\rho(\mathbbm{1}_{y^{-1}})$ for each $y \in \calG$.
Thus, given $\xi \in \CC \calG$ and $x,y \in \calG$, we have  
\begin{equation} 
\label{eq:CommuteWith1y}
\big( \overline{m}_f\overline{w}_\sigma \rho(\mathbbm{1}_{y^{-1}}) (\xi) \big)(x) 
= \big( \rho(\mathbbm{1}_{y^{-1}})\overline{m}_f\overline{w}_\sigma (\xi) \big)(x).
\end{equation}
The left-hand side (LHS) of \eqref{eq:CommuteWith1y} is zero whenever $x \notin \calG^V$ or  $\sou(\sigma^{-1}(x)) \neq \ran(y)$, and otherwise we get
\[
\Big( \overline{m}_f\overline{w}_\sigma \rho(\mathbbm{1}_{y^{-1}}) (\xi) \Big)(x) 
= f(x) \cdot \Big( \rho(\mathbbm{1}_{y^{-1}}) (\xi) \Big)(\sigma^{-1}(x)) 
= f(x) \xi\big( \sigma^{-1}(x) y \big).
\]
The right-hand side (RHS) of \eqref{eq:CommuteWith1y} is zero whenever $\sou(x) \neq \ran(y)$ or $xy \notin \calG^V$, and otherwise we get
\[
\big( \rho(\mathbbm{1}_{y^{-1}})\overline{m}_f\overline{w}_\sigma (\xi) \big)(x)
= \big( \overline{m}_f\overline{w}_\sigma(\xi) \big) (xy)
= f(xy) \xi\big( \sigma^{-1}(xy) \big).
\]

Fix $x,y\in\calG$ satisfying $\sou(x) = \ran(y)$ and $x\in \calG^V$, and note 
that $xy \in \calG^V$. 
Since LHS must be equal to RHS, 
setting $\xi := \delta_{\sigma^{-1}(xy)} \in \CC \calG$, the two sides of \eqref{eq:CommuteWith1y} are
\[
\mathrm{LHS}
= \begin{cases}
f(x) \delta_{\sigma^{-1}(xy)} (\sigma^{-1}(x)y) , &\text{if $\sou(\sigma^{-1}(x)) = \ran(y)$} \\
0, &\text{if $\sou(\sigma^{-1}(x)) \neq \ran(y)$} \\
\end{cases}
\andSep
\mathrm{RHS} = f(xy) 
\]
Since $f$ is nonzero on $\calG^V$, and since $xy \in \calG^V$, we get 
\begin{equation} 
\label{eq:Rel-sigma-inverse}
\sou(\sigma^{-1}(x)) 
= \ran(y) 
= \sou(x), \andSep
\sigma^{-1}(xy)
= \sigma^{-1}(x)y,
\end{equation}
and 
\begin{equation} 
\label{eq:Rel-f}
f(xy)
= f(x).
\end{equation}

For $x \in \calG^V$, taking $y=x^{-1}$ in \eqref{eq:Rel-f}, we get
\[
f(\ran(x))
= f(xx^{-1})
= f(x).
\]

For $x \in \calG^U$ and $y \in \calG$ with $\sou(x) = \ran(y)$, using that $\sigma$ is bijective and applying \eqref{eq:Rel-sigma-inverse}, we get
\begin{equation} 
\label{eq:Rel-sigma}
\sou(\sigma(x)) 
= \ran(y) 
= \sou(x), \andSep
\sigma(xy) = \sigma(x)y.
\end{equation}

For $x \in \calG^U$, applying \eqref{eq:Rel-sigma} with $\ran(x)$ in place of $x$, and $x$ in place of~$y$, we get
\begin{equation} 
\label{eq:Rel-sigma-general}
\sigma(x)
= \sigma(\ran(x)x)
= \sigma(\ran(x))x.
\end{equation}
and consequently 
\begin{equation} 
\label{eq:range-sigma}
\ran(\sigma(x))
= \ran\big( \sigma(\ran(x)) \big).
\end{equation}

Consider now the set 
\[
B_\sigma 
:= \big\{ \sigma(u) : u \in U \big\} 
\subseteq \calG.
\]

\textbf{Claim 1:} \emph{$B_\sigma $ is a bisection}.
To show that the restriction of the source map to~$B_\sigma$ is injective, let $u,v \in U$ satisfy $\sou(\sigma(u)) = \sou(\sigma(v))$.
Using \eqref{eq:Rel-sigma} at the second and fourth step, we get
\[
u 
= \sou(u) 
= \sou(\sigma(u))
= \sou(\sigma(v)) 
= \sou(v)
= v,
\]
and thus $\sigma(u)=\sigma(v)$.

To show that the restriction of the range map to $B_\sigma$ is injective, let $u,v \in U$ satisfy $\ran(\sigma(u)) = \ran(\sigma(v))$.
Then $\sigma(u)^{-1}\sigma(v)$ is well-defined and belongs to $\calG_U^U$.
Using that $\ran(\sigma(u)^{-1}\sigma(v)) = \sou(\sigma(v))$ at the second step, and using that $\sou(\sigma(u))=u$ by \eqref{eq:Rel-sigma} at the third step, we have:
\begin{align*}
\sigma\big( \sigma(u)^{-1}\sigma(v) \big) 
&\stackrel{\eqref{eq:Rel-sigma-general}}{=} \sigma\Big( \ran\big( \sigma(u)^{-1}\sigma(v) \big) \Big) \sigma(u)^{-1}\sigma(v) \\
&\ = \ \sigma\Big( \sou\big( \sigma(u) \big) \Big) \sigma(u)^{-1}\sigma(v) \\
&\ = \ \sigma(u) \sigma(u)^{-1} \sigma(v) 
= \sigma(v).
\end{align*}
Since $\sigma$ is bijective, we get $v = \sigma(u)^{-1} \sigma(v)$, from which it follows that $\sigma(u) = \sigma(v)$.
\vspace{.02cm}

Given $x \in \calG^U$, the unique element in $B_\sigma$ with source $\ran(x)$ is $\sigma(\ran(x))$. Using this at the first step, we get
\[
B_\sigma x
= \sigma(\ran(x))x
\stackrel{\eqref{eq:Rel-sigma-general}}{=} \sigma(x).
\]

\textbf{Claim 2:} \emph{$\sou(B_\sigma)=U$ and $\ran(B_\sigma)=V$.}
For the source of $B_\sigma$, using that $\sou(\sigma(u))=u$ for $u \in U$ by \eqref{eq:Rel-sigma}, we have
\[
\sou(B_\sigma) 
= \big\{ \sou(\sigma(u)) : u \in U \big\} 
= U.
\]
We turn to the range.
For $u \in U$, we have $\sigma(u) \in \calG^V$ and thus $\ran(\sigma(u)) \in V$, which shows that $\ran(B_\sigma) \subseteq V$.
On the other hand, given $v \in V$, let us consider $\sigma^{-1}(v) \in \calG^U$ and set $u := \ran(\sigma^{-1}(v)) \in U$.
Using \eqref{eq:range-sigma}, we get
\[
v 
= \ran(v)
= \ran(\sigma(\sigma^{-1}(v)))
= \ran(\sigma(\ran(\sigma^{-1}(v))))
= \ran(\sigma(u)) 
\in \ran(B_\sigma).
\]

\medskip

\textbf{Claim 3:} \emph{We have $B_\sigma =\supp(a)$}.
We will use Renault's $j$-map from \cref{pgr:Renault} to obtain a 
description of $\supp(a)$ in terms of $\sigma$.
For $x \in \calG$, we have
\begin{align}
\label{eq:ja}
\begin{split}
j(a)(x)
= j(\overline{m}_f \overline{w}_\sigma)(x) 
&= \big( \overline{m}_f \overline{w}_\sigma(\delta_{\sou(x)}) \big) (x) \\
&= \begin{cases}
f(x)\delta_{\sou(x)}(\sigma^{-1}(x)), &\text{ if $x \in \calG^V$} \\
0, & \text{ else}
\end{cases} \\
&= \begin{cases}
f(x), &\text{ if $x \in \calG^V$ and $\sou(x)=\sigma^{-1}(x)$} \\
0, & \text{ else}
\end{cases}.
\end{split}
\end{align}
Since $f$ is nonzero on $\calG^V$, it follows that
\[
\supp(a)
= \big\{ x \in \calG^V : \sou(x) = \sigma^{-1}(x) \big\}.
\]
We now show that the above set is equal to $B_\sigma$.
For $u \in U$, the element $x:=\sigma(u)$, which is in $B_\sigma$,
belongs to $\calG^V$. Then
\[
\sou(x)
= \sou(\sigma(u))
\stackrel{\eqref{eq:Rel-sigma}}{=} \sou(u)
= u
= \sigma^{-1}(x).
\]
This shows that $x \in \supp(a)$, and so $B_\sigma \subseteq \supp(a)$.
Conversely, let $x \in \calG^V$ satisfy $\sou(x)=\sigma^{-1}(x)$.
Then $\sigma^{-1}(x)$ is a unit and belongs to $\calG^U$, and thus $\sigma^{-1}(x) \in U$.
Then $x = \sigma(\sigma^{-1}(x)) \in B_\sigma$. This proves the claim.
\vspace{.2cm}

\textbf{Claim 4:} \emph{$B_\sigma$ is compact and open}.
It is clear that $B_\sigma$ is open, being the support of the continuous function $j(a)$ (see Claim~3).
On the other hand, since the image of $j(a)$ is contained in $\TT\cup\{0\}$, it follows that 
\[\supp(a)=\big\{x\in\calG\colon |j(a)(x)|>\tfrac{1}{2}\big\}.\]
Since $j(a)$ belongs to $C_0(\calG)$, the above set is compact, as desired.
\end{proof}

Ampleness will be used for the first time in the next lemma.

\begin{lma}
\label{prp:BasicHomotopyLemma}
Let $\calG$ be a Hausdorff, ample groupoid, let $p \in [1, \infty) \setminus \{2\}$, and let $a \in \PIMP(\Fp(\calG))$. Set $B := \supp(a)$.
Then $a \simeq \mathbbm{1}_B$ in $\PIMP(\Fp(\calG))$ .
\end{lma}
\begin{proof}
First of all if, if $X$ is a compact Hausdorff space and $f \colon X \to \TT$ is a continuous, non-surjective map, then $f$ is homotopic to the constant function $1$ via a homotomopy $[0,1] \to C(X,\TT)$ that is continuous for the $\|\cdot\|_\infty$-norm on  $C(X,\TT)$.
Indeed, if $z \in \TT$ is not contained in the image of $f$, using rotations we 
see that $f$ is homotopic to a continuous function $g\colon X\to \TT$ which does
not contain $1\in\TT$ in its image. By choosing an analytic branch of the logarithm, it is easy to check that $t\mapsto g^{1-t}$ is a homotopy between $g$ and the constant funciton $1$.

\medskip

Now, using Renault's $j$-map, we view $a$ as a function on $\calG$.
By \cref{prp:StructurePIMP-FpGpd}, the support $B$ is a compact, open bisection, and $a$ is $\TT$-valued on $B$.
Since $\calG$ is Hausdorff and ample, it follows that $B$ is a totally disconnected, compact, Hausdorff space.
We can therefore find disjoint open subsets $X, Y \subseteq B$ such that $B = X \cup Y$, and such that the restrictions $a|_{X} \colon X \to \TT$ and  $a|_{Y} \colon Y \to \TT$ are not surjective.

By the above, there are homotopies $t \mapsto f_t \in C(X,\TT)$ and $t \mapsto g_t \in C(Y,\TT)$ with $f_0 = a|_{X}$, $f_1 = \mathbbm{1}_{X}$, $g_0 = a|_{Y}$, and $g_1 = \mathbbm{1}_{Y}$.
For $t \in [0,1]$, define $h_t \colon \calG \to \TT$ by
\[
h_t(x) 
= \begin{cases}
f_t(x), &\text{if } x \in X \\
g_t(x), &\text{if } x \in Y \\
0, &\text{if } x \notin B \\
\end{cases}.
\]
We have $h_0=a$ and $h_1=\mathbbm{1}_B$, and each $h_t$ belongs to $\PIMP(\Fp(\calG))$ by \cref{prp:BisectionPIMP}.
For each $s,t \in [0,1]$, using that $h_s-h_t$ belongs to $C_c(\calG)$ and is supported on the bisection $B$, we have 
\[
\| h_s - h_t \|_\lambda = \| h_s - h_t \|_\infty 
\]
and it follows that $t \mapsto h_t \in \PIMP(\Fp(\calG))$ is continuous.
\end{proof}

\begin{thm} 
\label{prp:Bco_SFp}
Let $\calG$ be a Hausdorff, ample groupoid, and let $p \in [1, \infty) \setminus \{2\}$.
Then the map $\Phi_p\colon \Bco(\calG) \to \SInv(\Fp(\calG))$ from 
\cref{cor:PhiBisections} is an isomorphism of inverse semigroups.
\end{thm}
\begin{proof}
By \cref{prp:BisectionPIMP}, the map is a homomorphism of inverse semigroups.
Surjectivity follows by combining \cref{prp:StructurePIMP-FpGpd} with \cref{prp:BasicHomotopyLemma}.
To see injectivity, let $A, B \in \Bco(\calG)$ satisfy $\mathbbm{1}_A \simeq \mathbbm{1}_B$ inside $\PIMP(\Fp(\calG))$, say via a continuous path $t \mapsto a_t$.
By \cref{prp:StructurePIMP-FpGpd}, each $a_t$ takes values in $\TT\cup\{0\}$.
Using also that Renault's $j$-map is continuous, it follows that $\supp(a_t)$ does not change with $t$, and thus $A = \supp(a_0) = \supp(a_1) = B$.
\end{proof}

\section{Reconstructing groupoids} 
\label{sec:ReconstrGpds}

In this section, we prove that a Hausdorff, ample groupoid can be reconstructed from its reduced $L^p$-operator algebra, for any $p \in [1, \infty) \setminus \{2\}$;
see \cref{prp:LpReconstruction}.
We deduce an analogous result for the symmetrized \lpoa{s} in \cref{prp:SymLpRigidity}, and also for the $I$-norm completion of $C_c(\calG)$;
see \cref{prp:LIRigidity}.


\begin{thm}
\label{prp:LpReconstruction}
Let $\calG$ be a Hausdorff, ample groupoid, and let $p \in [1, \infty) \setminus \{2\}$. 
Then 
\[
\calG 
\cong \calG_{\mathrm{tight}} \big( \SInv(\Fp(\calG)) \big).
\]
\end{thm}
\begin{proof}
We will use Exel's reconstruction result from \cite{Exe10ReconstrTotDiscGpds}, where, given an inverse semigroup $S$, he defines the \emph{tight groupoid} $\calG_{\mathrm{tight}} (S)$ of $S$, as the groupoid of germs for the natural action of $S$ on the tight spectrum of the lattice of idempotents $E(S)$. Since
$\calG$ is an ample groupoid, it follows that $\calG \cong \calG_{\mathrm{tight}} (\Bco(\calG))$ by  \cite[Theorem~4.8]{Exe10ReconstrTotDiscGpds}.
We note that second-countability is not needed in the proof of \cite[Theorem~4.8]{Exe10ReconstrTotDiscGpds}.
Indeed, the reconstruction of the unit space $\calG^{(0)}$ as the tight spectrum of the lattice of idempotents $E(\Bco(\calG))$ holds in general;
see \cite[Theorem~3.6]{Exe10ReconstrTotDiscGpds}.
Applying \cite[Theorem~4.8]{Exe10ReconstrTotDiscGpds} at the first step, and \cref{prp:Bco_SFp} at the second, we get the desired isomorphisms
\[
\calG 
\cong \calG_{\mathrm{tight}} \big( \Bco(\calG) \big) 
\cong \calG_{\mathrm{tight}} \big( \SInv(\Fp(\calG)) \big).\qedhere
\] 
\end{proof}

\begin{cor}
\label{prp:LpRigidity}
Let $\calG$ and $\calH$ be Hausdorff, ample groupoids, and let $p \in [1, \infty) \setminus \{2\}$. 
Then:
\[
\calG \cong \calH \quad \text{if and only if} \quad
\Fp (\calG) \cong \Fp (\calH).
\]
\end{cor}
\begin{proof}
The forward implication is clear.
For the backwards implication, assume that there is an isometric isomorphism of \ba{s} $\Fp (\calG) \cong \Fp (\calH)$.
Then there is an isomorphism of inverse semigroups $\SInv(\Fp(\calG)) \cong \SInv(\Fp(\calH))$.
Using \cref{prp:LpReconstruction}, we obtain 
\[ 
\calG 
\cong \calG_{\mathrm{tight}}\big( \SInv(\Fp(\calG)) \big) 
\cong \calG_{\mathrm{tight}}\big( \SInv(\Fp(\calH)) \big) 
\cong \calH. \qedhere
\] 
\end{proof}

Next, we define the algebras of symmetrized $p$-pseudofunctions for an ample groupoid. These are Banach $\ast$-algebras that have been introduced and 
studied in \cite{AusOrt22GpdsHermBAlg}.
For group algebras, these have been considered in 
\cite{ElkPoo23arX:PropTBAlg, Elk24arX:SymmPseudofctLp}, and
feature prominently in \cite{SamWie20QuasiHermAmen, SamWie24ExoticCAlgsGeomGps}.

\begin{dfn} 
\label{dfn:SymFp}
Let $\calG$ be a Hausdorff, \etale{} groupoid, and let $p \in [1, \infty)$.
Define a norm on $C_c (\calG)$ by
\[
\|f \|_{p, \ast} 
:= \max \big\{ \| f \|_{\Fp} \, , \, \| f^\ast \|_{\Fp} \big\}
\]
for $f \in C_c (\calG)$. 
The \emph{symmetrized {$p$}-pseudofunction algebra}, denoted $\Fpast(\calG)$, is the Banach $\ast$-algebra obtained by completing $C_c (\calG)$ in this norm.
\end{dfn}

\begin{pgr}
Let $\calG$ be a Hausdorff, \etale{} groupoid.
We let $\LIG (\calG)$ denote the completion of~$C_c(\calG)$ with respect to the $I$-norm from \cref{pgr:ConvAlgs}.
It is not hard to see that the $I$-norm agrees with $\| \cdot \|_{1, \ast}$, and thus $\LIG (\calG) = F_{\lambda}^{1, \ast} (\calG)$.

Given $p \in (1,\infty)$ with dual H\"{o}lder exponent $q$ (such that $\tfrac{1}{p} + \tfrac{1}{q} = 1$), with the obvious modifications to \cite[Lemma~3.5]{AusOrt22GpdsHermBAlg}, an alternative description of the norm $\| \cdot \|_{p, \ast}$ is 
\begin{equation} 
\label{eq:SymFpNorm}
\|f \|_{p, \ast} 
= \max \big\{ \| f \|_{\Fp}, \| f \|_{F_{\lambda}^{q}} \big\},
\end{equation}
for $f \in C_c (\calG)$.
Note that $\Fpast(\calG) = F_{\lambda}^{q, \ast}(\calG)$. 
Moreover, the inequalities
\[
\| \cdot \|_{p, \ast} \geq \| \cdot \|_{F_{\lambda}^{q}}, \andSep
\| \cdot \|_{p, \ast} \geq \| \cdot \|_{\Fp}
\]
show that the identity on $C_c (\calG)$ extends to contractive homomorphisms 
\[
\varphi_p \colon \Fpast(\calG) \to \Fp(\calG), \andSep
\varphi_q \colon \Fpast(\calG) \to \Fq(\calG).
\]

Combining this with \eqref{eq:SymFpNorm}, we see that the diagonal map
\[
\varphi=\varphi_p\oplus \varphi_q \colon \Fpast(\calG) \to \Fp(\calG) \oplus \Fq(\calG)
\]
is an isometric homomorphism of Banach algebras. 
\end{pgr}

\begin{lma} 
\label{prp:FpStar-Fp-Injective}
Let $\calG$ be a Hausdorff, \etale{} groupoid and let $p \in [1, 2)$. 
Then the canonical contractive homomorphism $\varphi_p \colon \Fpast(\calG) \to \Fp(\calG)$ is injective.
\end{lma}
\begin{proof}
We treat the case $p=1$ first. 
Recall that $\LIG (\calG) = F_{\lambda}^{1,\ast}(\calG)$, and $F_{\lambda}^{1}(\calG) = \overline{\LIG (\calG)}^{\| \cdot \|_{F_{\lambda}^{1}}}$, so that $\varphi_1$ is just the canonical map $\LIG (\calG) \to F_{\lambda}^{1}(\calG)$, which is injective.

Assume that $1 < p < 2$. Let $a\in F_{\lambda}^{p,\ast}(\calG)$ and suppose that
$\varphi_p (a) = 0$.
Choose a sequence $(f_n)_{n\in\NN}$ in $C_c (\calG)$ such that 
$\|f_n- a\|_{p,\ast}\to 0$.
Since $\varphi_p$ is continuous and $\varphi_p(a)=0$, we get 
$\|f_n\|_{F^p_\lambda} \to 0$. 
In particular, fixing any $\xi \in \CC \calG$, we have $\|f_n \ast \xi\|_{\ell^p(\calG)} \to 0$. Upon passing to a subsequence if necessary, we may
assume that $f_n \ast \xi \to 0$ pointwise in $\calG$.

Note that $\|f_n- \varphi_q(a)\|_{F^q_\lambda}\to0$, and thus
$\|f_n \ast \xi - \varphi_q (a) \ast \xi\|_{\ell^q(\calG)}\to 0$. As before, upon passing to a subsequence if necessary,
we may assume that $f_n \ast \xi \to \varphi_q (a) \ast \xi$ pointwise in $\calG$.
Since pointwise limits are unique, we obtain that $\varphi_q (a) \ast \xi = 0$
for all $\xi\in \CC\calG$. 
Density of $\CC \calG$ in $\ell^q (\calG)$ forces $\varphi_q (a) = 0$. We deduce that $\varphi(a) = (\varphi_p (a) , \varphi_q (a)) = 0$, 
and since $\varphi$ is isometric, we conclude that $a = 0$, as desired.
\end{proof}

The next result is the analog of \cref{prp:PIMP-InvSgp-LPOA} for symmetrized $p$-pseudo\-function groupoid algebras. (Note that these algebras are not in general 
$L^p$-operator algebras, so that said proposition does not actually apply to them.)

\begin{prp}
\label{prp:PIMP-SymFP}
Let $\calG$ be a Hausdorff, \etale{} groupoid, and let $p \in [1, 2)$.
Then hermitian idempotents in $\Fpast(\calG)$ are ultrahermitian and commute.
It follows that $\PIMP(\Fpast(\calG))$ and $\SInv(\Fpast(\calG))$ are inverse semigroups.

Further, the natural map $\varphi_p \colon \Fpast(\calG) \to \Fp(\calG)$ induces isomorphisms of inverse semigroups
\[
\PIMP(\Fpast(\calG)) \cong \PIMP(\Fp(\calG)) \andSep
\SInv(\Fpast(\calG)) \cong \SInv(\Fp(\calG)).
\]
\end{prp}
\begin{proof}
The fact that hermitian idempotents in $\Fpast(\calG)$ are ultrahermitian and commute follows from \cref{prp:PIMP-InvSgp-LPOA} and \cref{prp:MapHermIdem},
using that the natural map $\Fpast(\calG) \to \Fp(\calG) \oplus \Fq(\calG)$ is isometric.
Therefore, $\PIMP(\Fpast(\calG))$ and $\SInv(\Fpast(\calG))$ are inverse semigroups by \cref{prp:PIMP-InvSgp} and \cref{prp:InvSgpHtpyPIMP}.

Since $\varphi_p \colon \Fpast(\calG) \to \Fp(\calG)$ is an approximately unital, contractive homomorphism, it preserves MP-partial isometries by \cref{rem:ContrHomPIMP} and thus induces a natural homomorphism $\Psi_p\colon \PIMP(\Fpast(\calG)) \to \PIMP(\Fp(\calG))$ of inverse semigroups, which
is injective by \cref{prp:FpStar-Fp-Injective}.
By \cref{prp:StructurePIMP-FpGpd}, every MP-partial isometry in $\Fp(\calG)$ is a $\TT\cup\{0\}$-valued map supported on a compact, open bisection, and in particular contained in the dense subalgebra $C_c(\calG)$.
It is immediate that such elements are also MP-partial isometries in $\Fpast(\calG)$, which shows that $\Psi_p$ is surjective and thus an isomorphism.

The isomorphism $\Psi_p$ is contractive and therefore induces a surjective homomorphism $\SInv(\Fpast(\calG)) \to \SInv(\Fp(\calG))$ of inverse semigroups.
To show injectivity of this map, we note that every homotopy in $\PIMP(\Fp(\calG))$ can be viewed as a $\|\cdot\|_\infty$-continuous homotopy in $C_c(\calG)$, which then is also continuous in $\PIMP(\Fpast(\calG))$ by arguments as in the proof of \cref{prp:BasicHomotopyLemma}. We omit the details.
\end{proof}

We obtain the analog of the reconstruction \cref{prp:LpReconstruction} and the rigidity \cref{prp:LpRigidity} in the context of symmetrized algebras.

\begin{thm}
\label{prp:SymLpRigidity}
Let $\calG$ be a Hausdorff, ample groupoid, and let $p \in [1, 2)$. 
Then there is a canonical groupoid isomorphism
$\calG \cong \calG_{\mathrm{tight}} \big( \SInv(\Fpast(\calG)) \big)$.
In particular, if $\calH$ is another Hausdorff, ample groupoid, then:
\[
\calG \cong \calH \quad \text{if and only if} \quad
\Fpast(\calG) \cong \Fpast(\calH).
\]
\end{thm}
\begin{proof}
Applying \cref{prp:LpReconstruction} at the first step, and using \cref{prp:PIMP-SymFP} at the second step, we get
\[
\calG 
\cong \calG_{\mathrm{tight}}\big( \SInv(\Fp(\calG)) \big) 
\cong \calG_{\mathrm{tight}}\big( \SInv(\Fpast(\calG)) \big).
\]

If $\calG$ and $\calH$ are isomorphic, then so are $\Fpast(\calG)$ and $\Fpast(\calH)$.
For the backwards implication, suppose that there is an isometric isomorphism of \ba{s} $\Fpast(\calG) \cong \Fpast(\calH)$.
Then there is an isomorphism of inverse semigroups $\SInv(\Fpast(\calG)) \cong \SInv(\Fpast(\calH))$.
Using the above reconstruction for $\calG$ and $\calH$, we get 
\[
\calG 
\cong \calG_{\mathrm{tight}} \big( \SInv(\Fpast(\calG)) \big) 
\cong \calG_{\mathrm{tight}} \big( \SInv(\Fpast(\calH)) \big) 
\cong \calH.\qedhere
\]
\end{proof}

Taking $p = 1$ in \cref{prp:SymLpRigidity} extends Wendel's classical result from \cite{Wen51IsometrIsoGpAlgs} from discrete groups to ample groupoids; we single out this special case.

\begin{cor} 
\label{prp:LIRigidity}
Let $\calG$ and $\calH$ be Hausdorff, ample groupoids. 
Then
\[
\calG \cong \calH \quad \text{if and only if} \quad
\LIG (\calG) \cong \LIG (\calH).
\]
\end{cor}


\providecommand{\href}[2]{#2}

\end{document}